\documentclass[a4paper, 11pt]{amsart}
\usepackage[utf8]{inputenc}
\usepackage[T2A]{fontenc}
\usepackage[english]{babel}
\usepackage{amssymb,amscd}

\usepackage[all]{xy}

\usepackage[a4paper,hmargin=2.5cm,vmargin=2.5cm]{geometry}

\usepackage{dsfont}
\usepackage{todonotes}
\usepackage{mathrsfs}
\usepackage{mathtools}
\usepackage{graphicx}

\usepackage{tikz}

\usepackage{hyperref}
\hypersetup{colorlinks=true, linkcolor=blue, citecolor=black!50!green}

\usepackage{bm}

\theoremstyle{plain}
\newtheorem{theorem}{Theorem}[section]
\newtheorem{lemma}[theorem]{Lemma}
\newtheorem{proposition}[theorem]{Proposition}

\theoremstyle{remark}

\theoremstyle{definition}
\newtheorem{example}[theorem]{Example}
\newtheorem{definition}[theorem]{Definition}
\newtheorem{remark}[theorem]{Remark}

\renewcommand{\AA}{\mathbb{A}}

\newcommand{\ZZ}{\mathbb{Z}}
\newcommand{\KK}{\mathbb{K}}

\newcommand{\G}{\mathbb{G}}
\newcommand{\GG}{G_{\overline{\chi}}}
\newcommand{\TT}{\mathbb{T}}
\newcommand{\ol}[1]{\overline{#1}}
\newcommand{\Alp}[1]{\pmb{\alpha_{#1}}}
\newcommand{\Bet}[1]{\pmb{\beta_{#1}}}
\renewcommand{\Im}{\mathop{\text{Im}}}

\DeclareMathOperator{\GL}{GL}

\DeclareMathOperator{\Ker}{Ker}

\DeclareMathOperator{\Ad}{Ad}

\DeclareMathOperator{\Aut}{Aut}

\DeclareMathOperator{\Spec}{Spec}
\DeclareMathOperator{\Hom}{Hom}

\DeclareMathOperator{\cone}{cone}

\begin{document}
	
	\title[ROOT MONOIDS AND ACTIVE ALGEBRAIC GROUPS]{Root monoids and active algebraic groups}
	\author{Ekaterina Nistiuk (Presnova) and Yulia Zaitseva}
	\address{HSE University, Faculty of Computer Science, Pokrovsky Boulvard 11, Moscow, 109028 Russia}
	\email{epresnova@hse.ru, nistiuk.k@gmail.com}
	\address{HSE University, Faculty of Computer Science, Pokrovsky Boulvard 11, Moscow, 109028 Russia}
	\email{yuliazaitseva@gmail.com}

	\thanks{The article was prepared within the framework of the project ``International academic cooperation'' HSE University.}
	
	\subjclass[2020]{Primary 20M32, 14M25, \ Secondary 14R20, 14R05}

\begin{abstract}
	We describe affine monoids whose group of invertible elements is an active semidirect product of a unipotent group and a torus, in terms of comultiplications on the algebra of regular functions. We introduce the notion of a root monoid, which is constructed from a set of Demazure root pairs on an affine toric variety, and study the properties of such monoids.
\end{abstract}

\maketitle

\section{Introduction}
Let $X$ be a normal irreducible algebraic variety, and let $X \times X \rightarrow X, (x, y) \mapsto x * y$ be a given morphism. Then $X$ is called an \textit{algebraic monoid} if for all $x, y, z \in X$ one has $x * (y * z) = (x * y) * z$ and there exists a point $1 \in X$ such that $x * 1 = 1 * x = x$. In this case, $1$ is called the \textit{neutral element}. The group of invertible elements $G(X)$ of an algebraic monoid $X$ is an algebraic group and is Zariski-open in~$X$. General properties of algebraic monoids can be found in~\cite{Br-1, Br-2, Pu, Re, Ri1, Ri2, Vin}.

We will be interested in affine algebraic monoids, i.e. those monoids whose underlying variety is affine. Let the base field $\KK$ be algebraically closed and of characteristic zero. In~\cite{ABZ}, commutative monoid structures on affine space $\mathbb{A}^n$ were studied. In particular, a classification of commutative monoid structures on $\AA^3$ was obtained, and in the recent preprint~\cite{AAZ} the classification in the noncommutative case was completed. In~\cite{DZ}, a description  for commutative monoids with group of invertible elements of corank $1$ on arbitrary normal affine varieties was obtained, that is, in the case when the codimension of the maximal torus $T$ in $G(X)$ is $1$. The case of noncommutative monoids on surfaces was considered in~\cite{Bil}. The work~\cite{YZ} generalizes the results of~\cite{Bil, DZ} to the case of a noncommutative group of invertible elements of corank $1$, i.e. for $G(X) \cong \G_a \leftthreetimes T$, where $\G_a = (\KK, +)$ is the additive group of the ground field. In particular, in~\cite{YZ} it was shown that $X$ is a toric variety, and a description of all affine monoids of corank~$1$ was obtained in terms of comultiplications $\KK[X] \rightarrow \KK[X] \otimes \KK[X]$. These comultiplications are of the form
$$\chi^u \mapsto \chi^u \otimes \chi^u (1 \otimes \chi^{e_1} + \chi^{e_2} \otimes 1 )^{\langle p, u \rangle },$$
where the toric variety $X$ corresponds to the cone $\sigma$, $p$ is a primitive vector on a ray of the cone~$\sigma$, and $e_1, e_2$ are Demazure roots corresponding to $p$.

In~\cite{YZ}, the notion of an \textit{active semidirect product} $G = U \leftthreetimes T$ was introduced, where $T$ is a torus and $U$ is the unipotent radical of the group $G$. Namely, $U \leftthreetimes T$ is called \textit{active} if $\dim T + \dim \Im \psi = \dim G$, where $\psi \colon T \rightarrow \Aut U$ is the homomorphism defining the semidirect product. In particular, the unipotent radical $U$ of an active group is commutative (Lemma~\ref{thm:commute}), and hence $G \cong \G_a^k \leftthreetimes T$. In the present paper we introduce the class of root monoids, describe affine monoids with active groups of invertible elements, and study the properties of these structures.

The paper is organized as follows. In Sections~\ref{sec:tor} and~\ref{root} we recall the necessary background on toric varieties, Demazure roots, and their relation with $\G_a$-actions. Section~\ref{aff} describes the structure of the group $\GG = \G_a^k \leftthreetimes T$, including an explicit formula for the comultiplication. The group $\GG$ is determined by a set of characters $\ol{\chi} = (\chi_1, \ldots, \chi_k)$ (see Lemma~\ref{lmYZ}).

In Section~\ref{sec:afftor} we introduce the notion of a \textit{root monoid}. This is a monoid structure on affine toric varieties with group of invertible elements $\GG$. The structure depends on the cone $\sigma$ of the corresponding toric variety $X_\sigma$, on a regular face $\tau \subset \sigma$, and on a set of Demazure roots compatible with $\tau$. A set of Demazure roots $\{e_1^{(r)}, e_2^{(r)} \mid r = 1,\ldots,k\}$ is said to be \emph{compatible} with the cone $\tau$ if for the primitive vectors $p_1, \ldots, p_k$ on the one-dimensional faces of the cone $\tau$ one has $\langle p_s, e_1^{(r)} \rangle = \langle p_s, e_2^{(r)} \rangle = -\delta_{rs}$, where $\delta_{rs}$ is the Kronecker symbol, $r, s = 1, \ldots, k$. 

\smallskip

More precisely, we prove the following theorem.

\begin{theorem}\label{thm111}
	Let $X = X_\sigma$ be an affine toric variety, where the cone $\sigma$ contains a $k$-dimensional regular face $\tau \subset \sigma$, and let $\{e_1^{(r)}, e_2^{(r)} \mid r = 1,\ldots,k \}$ be a set of Demazure roots compatible with~$\tau$. Then the map $\KK[X] \rightarrow \KK[X] \otimes \KK[X]$ given by 
	\begin{equation}\label{thm:finalcom1}
		\chi^u \mapsto \chi^u \otimes \chi^u \prod_{r = 1}^{k} (1 \otimes \chi^{e_1^{(r)}}  + \chi^{e_2^{(r)}} \otimes 1)^{\langle p_r, u\rangle},
	\end{equation}
	where $\KK[X_\sigma] = \bigoplus_{u \in S_\sigma} \KK \chi^u$, defines a comultiplication that determines a monoid structure on $X$ with group of invertible elements~$\GG$, where $\chi_r = \chi^{e_2^{(r)} - e_1^{(r)}}$. 
\end{theorem}
The monoids described in the theorem are called \emph{root monoids}. This construction covers a number of previously known results and includes, as special cases, the monoids obtained in~\cite{DZ,Bil,YZ}.

In Section~\ref{sec:actmon} we obtain a description of affine monoids with active group of invertible elements.
\begin{theorem}
	Every affine algebraic monoid with an active group of invertible elements is isomorphic to a root monoid for some cone~$\sigma$, its $k$-dimensional regular face~$\tau$, and a set of Demazure roots \[\{e_1^{(r)}, e_2^{(r)} \mid r = 1, \ldots, k\},\] compatible with $\tau$. Moreover, the group of invertible elements of a root monoid is active if and only if the differences $e_2^{(r)} - e_1^{(r)}$ are linearly independent.
\end{theorem}

In Section~\ref{sec:ex} we specialize the obtained result to the case of affine space. The structure of a root monoid is determined by $2k(n-k)$ nonnegative integers. The second example describes monoid structures on the cylinder over a quadratic cone. 
In Section~\ref{sec:idemp} we consider the decomposition of the root monoid $X_\sigma$ into a union of orbits $O_\gamma$ with respect to the acting torus, and for each face~$\gamma$ of the cone~$\sigma$ we describe the set~$E_\gamma$ of all idempotents in the orbit $O_\gamma$. A nontrivial answer appears in the case of those orbits~$O_\gamma$ for which, for every ray $p_r$ not belonging to $\gamma$, exactly one of the corresponding Demazure roots $e_1^{(r)}, e_2^{(r)}$ lies in $\gamma^\perp$.
In Section~\ref{sec:idemp_geom} we show that $E_\gamma$ is a $T$-orbit, and the closure of $E_\gamma$ is a unipotent group orbit consisting of sets of the form $E_{\gamma'}$ for a collection of faces~$\{\gamma'\}$ depending on $\tau$ and on the Demazure roots~$e_1^{(r)}, e_2^{(r)}$, see Theorem~\ref{idem_thm_2}.
Finally, Section~\ref{sec:center} is devoted to the description of the center of an active monoid~$X_\sigma$. It is contained in the closure of the orbit $O_\tau$, and we find equations defining the center of the active monoid in $\ol{O_\tau}$. In the case of monoids of corank~1 studied in~\cite{YZ}, every root monoid is either commutative or active. In the case of corank greater than~1, intermediate situations are possible. 

The authors express their gratitude to Ivan Arzhantsev for helpful remarks and constant support, and thank Sergey Gorchinskiy and Konstantin Shramov for noticing that active monoids are toric varieties.

\section{Preliminaries}
\subsection{Toric varieties}\label{sec:tor}
In this section we introduce notation and briefly describe the main results on the structure of affine toric varieties. A more detailed exposition can be found in~\cite{F}.

Let $N$ be a lattice of rank $n$, and let $N_{\mathbb{Q}} = N \otimes_{\mathbb{Z}} \mathbb{Q}$ be the corresponding rational vector space of dimension $n$. Every strongly convex polyhedral cone $\sigma \subseteq N_{\mathbb{Q}}$ corresponds to an affine toric variety $X_\sigma$.

Recall that $M = \text{Hom}(N, \mathbb{Z})$ is the dual lattice, $M_{\mathbb{Q}} = M \otimes_{\mathbb{Z}} \mathbb{Q}$ is the corresponding rational $n$-dimensional vector space, and $\langle \cdot, \cdot \rangle\colon N_{\mathbb{Q}} \otimes M_{\mathbb{Q}} \rightarrow \mathbb{Q}$ is the natural pairing. We identify elements of the lattice $M$ with characters of the torus $\TT = \text{Hom}(M, \KK^\times)$. The character corresponding to $u \in M$ is denoted by $\chi^u\colon \TT \rightarrow \KK^\times$.

Given the cone $\sigma$, we construct the dual cone $$\sigma^\vee = \{u \in M_{\mathbb{Q}} \mid \langle v, u \rangle \geq 0 \text{ for all } v \in \sigma\} \subseteq M_{\mathbb{Q}}.$$

Then $S_\sigma = \sigma^\vee \cap M$ is a finitely generated semigroup, $\KK[S_\sigma] = \bigoplus_{u \in S_\sigma} \KK \chi^u$ is a finitely generated $\KK$-algebra, the toric variety $X_\sigma$ is the spectrum of this algebra, and $\TT$ is the acting torus. Moreover, points of the variety $X_\sigma$ correspond to semigroup homomorphisms $S_\sigma \rightarrow \KK$, where $\KK$ is regarded as a semigroup under multiplication. Indeed, a semigroup homomorphism $S_\sigma \rightarrow \KK$ determines an algebra homomorphism $\KK[X_\sigma] \rightarrow \KK$, which corresponds to a point of $X_\sigma$.

A cone $\sigma$ is called regular if the primitive vectors $p_1, \ldots, p_k$ on its rays form part of a basis of the lattice $N$. In this case, $X_\sigma$ is isomorphic to the direct product of an affine space of dimension~$k$ and a torus of dimension $n-k$.

It is well known that the orbits of the torus $\TT$ acting on the variety $X_\sigma$ correspond to the faces of the cone~$\sigma$. Moreover, for each face $\tau \subseteq \sigma$ there exists a distinguished point $x_\tau \in X_\sigma$ corresponding to the following semigroup homomorphism:
\begin{equation}\label{tor:fun_orb}
	u \mapsto \begin{cases} 
		1, \text{ if } u \in \tau^\perp\\
		0 \text{ otherwise.} 
	\end{cases}
\end{equation} 

The orbit of the point $x_\tau$ under the torus action is denoted $O_\tau$. In terms of semigroup homomorphisms, $O_\tau$ consists exactly of those homomorphisms $S_\sigma \rightarrow \KK$ such that 
\begin{equation}\label{tor:fun_orb2}
	u \mapsto \begin{cases} 
		\KK^\times, \text{ if } u \in \tau^\perp\\
		0 \text{ otherwise.} 
	\end{cases}
\end{equation}

It is known that in this way one can describe all torus orbits on the toric variety $X_\sigma$.

Every semigroup homomorphism $S_\sigma \to S_{\sigma'}$ defines an algebra homomorphism $\KK[S_\sigma] \to \KK[S_{\sigma'}]$ and a morphism $\Spec \KK[S_{\sigma'}] \to \Spec \KK[S_\sigma]$. In particular, if $\tau \subseteq \sigma$, then $S_\tau \supseteq S_\sigma$, which determines a morphism $X_\tau \to X_\sigma$. One can show that the morphism $X_\tau \to X_\sigma$ is an open embedding if and only if $\tau$ is a face of the cone~$\sigma$. Moreover, in this case $S_\tau = S_\sigma + \ZZ_{\ge0}(-u')$, where $u' \in M$ and $\tau = \sigma \cap u'^\perp$, or equivalently, $u' \in M$ belongs to the relative interior of the dual face $\tau^\perp \cap \sigma^\vee$ of the cone~$\sigma^\vee$. Then $\KK[S_\tau]$ is the localization of $\KK[S_\sigma]$ by $\chi^{u'}$, which defines the embedding $X_\tau \subseteq X_\sigma$ as a principal open subset. In terms of semigroup homomorphisms, $X_\tau$ consists of those points $S_\sigma \to \KK$ such that the image of $u'$ is nonzero, or equivalently, such that the image of any $u' \in M$ from the relative interior of $\tau^\perp \cap \sigma^\vee$ is nonzero. In particular, for all faces $\tau$ of the cone~$\sigma$ the torus actions on $X_\tau$ as toric varieties are compatible and have the open orbit~$X_0$.

\subsection{Derivations and Demazure roots}\label{root}

Let $A$ be an algebra. A linear operator $\delta\colon A \rightarrow A$ is called a \textit{derivation} if it satisfies the Leibniz rule:
$\delta(fg) = \delta(f) g + f \delta(g)$ for all $f,g \in A$.

A derivation $\delta$ is called \textit{locally nilpotent} (LND) if for any element $f \in A$ there exists a positive integer $n \in \ZZ_{>0}$ such that $\delta^n(f) = 0$. It is known that the exponential map establishes a bijection between the set of locally nilpotent derivations and rational $\G_a$-actions on the algebra~$A$. In particular, an LND $\delta$ corresponds to a $\G_a$-action $\varepsilon\colon \G_a \times A \rightarrow A$, $(\alpha, f) \mapsto \alpha\cdot f = \exp(\alpha \delta)(f)$, where $\alpha \in \G_a$, $f \in A$. Moreover, for an affine algebraic variety $X$ there is a bijection between $\G_a$-actions on $X$ and $\G_a$-actions on $\KK[X]$.

Now suppose that the algebra $A$ is graded, $A = \bigoplus_{u \in S} A_u$, where $S$ is a semigroup. A derivation~$\delta$ is called \textit{homogeneous} if it maps homogeneous elements to homogeneous ones. Then the \textit{degree} of a homogeneous derivation is defined as an element $\deg \delta \in \ZZ S$, namely, for any $u \in S$ one has $\delta(A_u) \subseteq A_{u + \deg \delta}$.

Let $X_\sigma$ be an affine toric variety, and let $p_1, \ldots, p_m \in N$ be the primitive vectors on the one-dimensional faces of the cone $\sigma$. For each $1 \leq i \leq m$ define the set
\begin{equation*}
	\mathfrak{R}_i  = \{ e \in M \mid \langle p_i, e \rangle = -1 , \, \langle p_j, e \rangle \geq 0 \text{ for all } j \neq i,\, 1 \leq j \leq m \}.
\end{equation*}

The elements of the set $\mathfrak{R} = \bigsqcup_{1 \leq i \leq m} \mathfrak{R}_i$ are called \textit{Demazure roots}. It is easy to check that each set $\mathfrak{R}_i \subset M$ is nonempty. It is known that Demazure roots are in one-to-one correspondence with nonzero homogeneous locally nilpotent derivations of the algebra $\KK[X_\sigma]$, up to proportionality. Moreover, a Demazure root $e$ is the degree of the corresponding homogeneous LND. In turn, homogeneous LNDs correspond to $\G_a$-actions normalized by the acting torus~$\TT$.

\section{Semidirect products}\label{aff} 

\subsection{Definition of the group $\GG$}

Let $G$ be a connected solvable affine algebraic group. Then $G = U \leftthreetimes T$, where $T$ is a torus and $U$ is the unipotent radical of $G$. Let $\psi\colon T \rightarrow \Aut U$ be the homomorphism defining the semidirect product. It is known that $U$ is isomorphic to an affine space $\AA^k$ as an affine variety.

\begin{lemma}\cite[Lemma 2]{YZ}\label{lmYZ} One can introduce coordinates on $U$ in such a way that $\psi(t) = \mathrm{diag}(\chi_1(t), \ldots, \chi_k(t)) \in \GL(U)$, where $\chi_1(t), \ldots, \chi_k(t)$ are characters of the torus $T$. 
\end{lemma}

In this section we consider the case of a commutative unipotent radical, i.e. $U \cong \G_a^k$. Under the assumptions of Lemma~\ref{lmYZ}, we denote by $\ol{\alpha} = (\alpha_1, \ldots, \alpha_k) \in U$ the corresponding coordinates, and by $\GG$ the group, where $\ol{\chi} = (\chi_1, \ldots, \chi_k)$ is the set of characters. Denote $\beta_r = \chi_r(t)^{-1}\alpha_r$ for $r = 1, \ldots, k$. Then every element of the group $\GG = \G_a^k \leftthreetimes T$ can be represented as a product of elements from $U$ and $T$ in two ways: 
\begin{equation*}
g = \ol{\alpha} \cdot t = t \cdot \ol{\beta},
\end{equation*}
where $\ol{\beta} = (\beta_1, \ldots, \beta_k)$. 

The multiplication on $\GG$ is given by
\begin{equation}\label{def:mult}
	(\ol{\alpha}\cdot t) \cdot (\ol{\alpha'}\cdot t') = (\ol{\alpha} + \ol{\chi}(t)\ol{\alpha'}) \cdot tt',
\end{equation}
\begin{equation}
	(t \cdot \ol{\beta}) \cdot (t' \cdot \ol{\beta'}) = tt' \cdot (\ol{\beta} + \ol{\chi}^{-1}(t)\ol{\beta'}).
\end{equation}
Here the set of values of the $k$ characters is multiplied coordinatewise by an element of $\G_a^k$. We denote by $\Alp{r}, \Bet{r} \in \KK[\GG]$ the coordinate functions, where $\Alp{r}(g) = \alpha_r$ and $\Bet{r}(g) = \beta_r$.

\subsection{Cones and homogeneity of coordinate spaces}\label{sec:conhom}
Let $n = \dim \GG$. As an algebraic variety, $\GG$ is the direct product of the affine space~$\AA^k$ and the torus~$T$ of dimension $n-k$. That is, $\GG$ is isomorphic to an affine toric variety with a regular $k$-dimensional cone $\tau$ in the $n$-dimensional lattice. Recall that the algebra of regular functions is expressed as
\begin{equation}\label{1grad}
	\KK[\GG] = \bigoplus_{u \in S_\tau} \KK \chi^u.  
\end{equation}
Consider the map $\theta\colon T \times T \rightarrow \Aut(\GG)$ defined as follows: an element $(t_1, t_2) \in T \times T$ acts on an element $g \in \GG$ by $t_1 g t_2^{-1}$. Specifically, for an element $\ol{\alpha}\cdot t$ we have
\begin{equation}\label{1eq:act}
	\theta(t_1, t_2)(\ol{\alpha}\cdot t) = t_1\cdot \ol{\alpha}\cdot t\cdot t_2^{-1} = \ol{\chi}(t_1)\ol{\alpha}\cdot t_1 t t_2^{-1}.
\end{equation}
Denote $\TT = \Im \theta$. Note that $\Ker \theta = \text{diag} (\cap \Ker \chi_j) \subset T \times T$, hence the value of the character $\chi_r(t_1)$ does not depend on the choice of representative $\theta(t_1, t_2) = \theta(t_1', t_2') \in \TT$, so the map $\theta(t_1, t_2) \mapsto \chi_r(t_1)$ defines a character of the torus~$\TT$. Consequently, all coordinate functions $\Alp{r}$ are homogeneous with respect to the grading~\eqref{1grad}, and thus, multiplying by a scalar, we can asssume $\Alp{r} = \chi^{a_r}$, where $a_r \in \tau^\vee$.
Similarly, we denote $\Bet{r} = \chi^{b_r}$, where $b_r \in \tau^\vee$. Moreover, since $\Bet{r} = \chi_r^{-1} \Alp{r}$, all characters $\chi_r^{-1}$ of the torus~$T$, which also correspond to characters of the torus $\TT$, can be recovered from $a_r, b_r \in M$, namely, $\chi_r = \chi^{a_r - b_r}$.

\subsection{Torus orbits}\label{1sec:orb}
Note that the set \[O_r = \{ \ol{\alpha}\cdot t \in \GG \mid \alpha_r = 0, \; \alpha_j \neq 0 \text{ for } j \neq r \}\]
is an $(n-1)$-dimensional orbit of the torus $\TT$ under the action~\eqref{1eq:act}. The codimension of such an orbit is $1$, hence $O_r$ corresponds to a one-dimensional face $\rho_r$ of the cone~$\tau$, i.e. a ray. Let $p_r$ be the primitive vector on the corresponding ray, and denote the orbit by $O_{p_r}$. Since the cone $\tau$ is regular, all primitive vectors $p_r$ form part of a basis of the lattice $N$.

We have already shown that the coordinate functions $\Alp{r}$ correspond to characters $\chi^{a_r}$ of the torus $\TT$. Since the coordinate function $\Alp{r}$ vanishes at all points of the orbit $O_{p_r}$, by the correspondence~\eqref{tor:fun_orb}, we have $a_r \notin \rho_r^\perp$, hence $\langle p_r, a_r \rangle \neq 0$. Moreover, $\langle p_r, a_r \rangle > 0$ since ${a_r \in \tau^\vee \subset \rho_r^\vee}$. On the other hand, the function $\Alp{r}$ does not vanish on points of the orbit $O_{p_j}$, so $a_r \in \rho_j^\perp$ for $j \neq r$, which means that $\langle p_j, a_r \rangle = 0$.

Next, note that by formula~\eqref{1eq:act} the set 
\[O_\tau = \{ \ol{\alpha} \cdot t \in  \GG \mid \alpha_r = 0 \;\; \forall r\} = T\] is an $(n-k)$-dimensional orbit under the action of~$\TT$. Consider the sublattice ${M(\tau) = \tau^\perp \cap M}$. Then the torus $\text{Hom}(M(\tau), \KK^\times)$ can be identified with the torus $T = O_\tau$, hence $\KK[T] = \KK[M(\tau)].$

Since $\KK[\GG] = \KK[T][\Alp{1}]\ldots[\Alp{k}]$, the semigroup $S_\tau$ is generated by the sublattice $M(\tau)$ together with all $a_r \in M$. Then any element $u \in S_\tau$ can be written as the sum of an element from $\tau^\perp$ and an integer linear combination of the elements $a_r$. Consequently, $\langle p_r, u \rangle$ is divisible by $\langle p_r, a_r \rangle$, which means that $S_\tau$ is not a saturated semigroup when $\langle p_r, a_r\rangle > 1$. Thus, $\langle p_r, a_r \rangle = 1$.

Similarly, $\langle p_r, b_r \rangle = 1$. Hence, $a_r - b_r \in \tau^\perp$.	

\begin{proposition}\label{1l:comult} 
The comultiplication $\KK[\GG] \rightarrow \KK[\GG] \otimes \KK[\GG]$ is given by
\begin{equation}\label{1l:comu}
	\chi^u \mapsto \chi^u \otimes \chi^u \prod_{r = 1}^{k} (1 \otimes \Alp{r}^{-1} + \Bet{r}^{-1} \otimes 1)^{\langle p_r, u\rangle},
\end{equation}  
where $u \in S_\tau$, $\KK[\GG] = \bigoplus\limits_{u \in S_\tau} \KK \chi^u.$  
\end{proposition}

\begin{proof}
Recall that the pairing $\langle\cdot,\cdot\rangle$ is linear in each argument, and for characters we have $\chi^{u_1+u_2} = \chi^{u_1}\chi^{u_2}$. Therefore, it is enough to verify formula~\eqref{1l:comu} for generators of the semigroup~$S_\tau$, i.e., for $u \in M(\tau)$ and for $u = a_r$.

Since we previously identified $\Hom(M(\tau), \KK^\times)$ with the torus $T \subseteq \GG$, for all $u \in M(\tau)$, i.e. for the characters $\chi^u \in \KK[T]$, the comultiplication takes the form $\chi^u \mapsto \chi^u \otimes \chi^u$. This coincides with formula~\eqref{1l:comu} because $\langle p_r, u \rangle = 0$.

Now let us check the case $u = a_r$, i.e. $\chi^u = \chi^{a_r} = \Alp{r}$. Since multiplication in the group $\GG$ is given by formula~\eqref{def:mult}, we have
$$\Alp{r} \mapsto \Alp{r} \otimes 1  + \chi_r(t)\otimes \Alp{r} =\Alp{r} \otimes \Alp{r} (1 \otimes \Alp{r}^{-1}  + \chi_r(t) \Alp{r}^{-1}\otimes 1).$$
To complete the proof, it remains to note that $\chi_r(t)\Alp{r}^{-1} = \Bet{r}^{-1}$ and $\langle p_j, a_r \rangle = \delta_{rj}$, where $\delta_{rj}$ is the Kronecker symbol.
\end{proof}

\section{Root monoids}\label{sec:afftor}
\subsection{Classification of root monoids}
In this section we introduce the structure of a monoid with group of invertible elements $\GG$ on affine toric varieties.

\begin{definition}
Let $\tau$ be a regular $k$-dimensional face of the cone $\sigma$, and let $p_1, \ldots, p_k$ be the primitive vectors on the one-dimensional faces of $\tau$. We call a set of Demazure roots of the cone~$\sigma$ $$\{e_1^{(r)}, e_2^{(r)} \mid r = 1,\ldots,k\}$$ \textit{compatible with the face $\tau$} if $\langle p_s, e_1^{(r)} \rangle = \langle p_s, e_2^{(r)} \rangle = -\delta_{rs}$, where $\delta_{rs}$ is the Kronecker symbol, $r, s = 1, \ldots, k$.
In particular, this condition implies that ${e_1^{(r)}, e_2^{(r)} \in \mathfrak{R}_r}$.
\end{definition}

In Section~\ref{1sec:orb} we identified the torus $T$ with $\text{Hom}(M(\tau), \KK^\times)$. If Demazure roots are compatible with the face $\tau$, then $e_2^{(r)} - e_1^{(r)} \in \tau^\perp$, and therefore the characters $\chi_r := \chi^{e_2^{(r)} - e_1^{(r)}}$ of the torus $\TT$, which acts on the variety $X_\sigma$, can be regarded as characters of the torus $T$. Thus, a compatible set of Demazure roots determines the structure of a semidirect product $\GG = \G_a^k \leftthreetimes T$ by Lemma~\ref{lmYZ}. 

\begin{theorem}\label{thm:mon}
Let $X = X_\sigma$ be an affine toric variety, the cone $\sigma$ contain a $k$-dimensional regular face $\tau \subset \sigma$, and $\{e_1^{(r)}, e_2^{(r)} \mid r = 1,\ldots,k\}$ be a set of Demazure roots compatible with $\tau$. Then the map $\KK[X] \rightarrow \KK[X] \otimes \KK[X]$ given by
\begin{equation}\label{thm:finalcom}
	\chi^u \mapsto \chi^u \otimes \chi^u \prod_{r = 1}^{k} (1 \otimes \chi^{e_1^{(r)}} + \chi^{e_2^{(r)}} \otimes 1)^{\langle p_r, u\rangle},
\end{equation}
where $\KK[X_\sigma] = \bigoplus_{u \in S_\sigma} \KK \chi^u$, defines a comultiplication that determines a monoid structure on~$X$ with group of invertible elements~$\GG$, where $\chi_r = \chi^{e_2^{(r)} - e_1^{(r)}}$. Moreover, the group of invertible elements coincides with the open subset $X_\tau$.
\end{theorem}

\begin{definition}
The monoids constructed in Theorem~\ref{thm:mon} are called \textit{root monoids}.
\end{definition}
 
\begin{remark}
Recall that a monoid $X$ is called commutative if the multiplication on $X$ is commutative, which is equivalent to commutativity of the group of invertible elements. The group $\GG$ is commutative if and only if $\ol{\chi} = \ol{0}$. Hence, the root monoid $X$ is commutative if and only if $e_1^{(r)} = e_2^{(r)}$ for all $r = 1, \ldots, k.$
\end{remark}

\subsection{Proof of Theorem~\ref{thm:mon}}

First, let us show that formula~\eqref{thm:finalcom} indeed defines a comultiplication on $\KK[X]$.

Introduce some notation. Let $\ol{p} = (p_1, \ldots, p_k)$ be the set of primitive vectors on the rays of the cone $\tau$. Define the generalized scalar product $\langle \ol{p}, u \rangle = (\langle p_1, u \rangle, \ldots,  \langle p_k, u \rangle)$. 
We also denote by $\ol{i} = (i_1, \ldots, i_k) \in \ZZ^k$ a tuple of nonnegative integers, and by
${ \langle \ol{p}, u \rangle \choose \ol{i} }  =  { \langle p_1, u \rangle \choose i_1} \cdot \ldots \cdot  { \langle p_k, u \rangle \choose i_k}$
the product of binomial coefficients. Then formula~\eqref{thm:finalcom} can be rewritten as
\begin{equation}\label{1fin_mult}
	\chi^u \mapsto \chi^u \otimes  \chi^u  \prod_{r = 1}^{k} (1 \otimes \chi^{e_1^{(r)}} + \chi^{e_2^{(r)}} \otimes 1)^{\langle p_r, u\rangle} = \sum \limits_{\ol{i} + \ol{j} = \langle \ol{p}, u \rangle } { \langle \ol{p}, u \rangle \choose \ol{i} } \chi^{u + \ol{i}\cdot \ol{e_2}} \otimes \chi^{u + \ol{j}\cdot \ol{e_1}}.
\end{equation}

To check that the result of this operation belongs to $\KK[X] \otimes \KK[X]$, we need to verify that $u + \ol{i}\cdot \ol{e_2} = u + i_1 e_2^{(1)} + \ldots + i_k e_2^{(k)} = u + \sum\limits_{r = 1}^k i_r e_2^{(r)} \in S_\sigma$, which is equivalent to $\langle p_s, u + \ol{i}\cdot \ol{e_2} \rangle \geq 0$ and $\langle q, u + \ol{i}\cdot \ol{e_2} \rangle \geq 0$ for the remaining primitive vectors $q$ on the rays of the cone~$\sigma$. Note that since $e_2^{(r)}$ is a Demazure root compatible with $\tau$, we have $\langle p_r, e_2^{(r)} \rangle = -1$, $\langle p_s, e_2^{(r)} \rangle = 0$ for $r \neq s$, and its pairing with the other primitive vectors on the rays of~$\sigma$ is nonnegative. Let us first check $\langle p_s, u + \ol{i}\cdot \ol{e_2} \rangle \geq 0$. In formula~\eqref{1fin_mult} the summation runs over those $\ol{i}$ such that $i_s \leq \langle p_s, u \rangle$. Hence, $\langle  p_s, u + \ol{i}\cdot \ol{e_2} \rangle = \langle p_s, u \rangle + \langle p_s,  i_s e_2^{(s)} \rangle = \langle  p_s, u \rangle - i_s \geq 0$. The second condition $\langle q, u + \ol{i}\cdot \ol{e_2} \rangle \geq 0$ also holds, because $\langle q, u\rangle \geq 0$ and $\langle q, i_re_2^{(r)}\rangle = i_r \langle q, e_2^{(r)}\rangle \geq 0$ for all~$r$. Thus, $\chi^{u + \ol{i}\cdot \ol{e_2}} \in \KK[X]$, and similarly $\chi^{u + \ol{j}\cdot \ol{e_1}} \in \KK[X]$.

Now let us check the existence of a neutral element. For a cone $\tau \subset \sigma$ there exists (see Section~\ref{sec:tor}) a distinguished point $x_\tau \in X = X_\sigma$ such that
\begin{equation}
	\begin{cases} 
		\chi^u(x_\tau) = 1, \text{ if } u \in \tau^\perp\\
		\chi^u(x_\tau) = 0 \text{ otherwise}.
	\end{cases}
\end{equation}

Apply formula~\eqref{1fin_mult} to a pair of points $x, y \in X$:
\begin{equation}
	\chi^u (x * y) = \sum \limits_{\ol{i} + \ol{j} = \langle \ol{p}, u \rangle } { \langle \ol{p}, u \rangle \choose \ol{i} } \chi^{u + \ol{i}\cdot \ol{e_2}}(x) \cdot  \chi^{u + \ol{j}\cdot \ol{e_1}} (y).
\end{equation}

In $\chi^u (x_\tau * x)$, the only nonzero terms are those for which $u + \ol{i}\cdot \ol{e_2} = u + i_1 e_2^{(1)} + \ldots + i_k e_2^{(k)} \in \tau^\perp$, which is equivalent to $\langle p_s, u + \ol{i}\cdot \ol{e_2} \rangle = 0$ for all $s = 1, \ldots, k$. Since $\langle p_s, e_2^{(r)}\rangle = -\delta_{rs}$, we obtain $\langle p_s, u + \ol{i}\cdot \ol{e_2} \rangle = \langle p_s, u + i_s e_2^{(s)} \rangle = 0$, which is equivalent to $i_s = \langle p_s, u \rangle$. Hence, exactly one term remains, corresponding to $\ol{i} = \langle \ol{p}, u \rangle$. Therefore,
\begin{equation}
	\chi^u (x_\tau * x) =  \chi^{u + \langle \ol{p}, u \rangle\cdot \ol{e_2}}(x_\tau) \cdot  \chi^{u}(x) =  \chi^{u} (x).
\end{equation}

Similarly, $\chi^u (x * x_\tau) = \chi^{u} (x)$, so $x_\tau$ is the neutral element. 

Now let us show that the set of invertible elements of the monoid $X_\sigma$ coincides with $X_\tau$. An element $x \in X_\sigma$ is invertible if there exists $y \in X_\sigma$ such that $x * y = x_\tau$, i.e. $\chi^u (x_\tau) =  \chi^u (x * y)$ for all $u \in S_\sigma$. Take $u$ from the relative interior of $\tau^\perp \cap \sigma^\vee$. Then $\langle p_r, u \rangle = 0$ for all $r = 1, \ldots, k$ and $\chi^u(x_\tau) = 1$. Hence,
\begin{equation}
	1 = \chi^u (x_\tau) =  \chi^u (x * y) = \chi^u (x)  \chi^u (y)  \prod_{r = 1}^{k} (\chi^{e_1^{(r)}}(y) + \chi^{e_2^{(r)}} (x))^{\langle p_r, u\rangle} = \chi^u (x)  \chi^u (y).
\end{equation}

Thus, $\chi^u(x) \in \KK^\times$. Using~\eqref{tor:fun_orb2}, we obtain $x \in X_\tau$.

Note that the variety $X_\tau$ is isomorphic to the product of a $k$-dimensional affine space and an $(n-k)$-dimensional torus. Therefore, we can identify $X_\tau$ with the group $\GG \cong \G_a^k \leftthreetimes T$, where $\chi_r = \chi^{e_2^{(r)} - e_1^{(r)}}$. Indeed, since $\langle p_s, e_1^{(r)} \rangle = \langle p_s, e_2^{(r)} \rangle = -\delta_{rs}$, and because $- e_1^{(r)}, - e_2^{(r)} \in \tau^\vee$, the characters $\Alp{r} := \chi^{- e_1^{(r)}}$ and $\Bet{r} := \chi^{- e_2^{(r)}}$ are regular functions on $X_\tau$. Then the restriction of comultiplication~\eqref{thm:finalcom} to the open subset $X_\tau$ coincides with formula~\eqref{1l:comu} for $\chi_r = \Alp{r}/\Bet{r}$. Thus, the dual map $X_\tau \times X_\tau \to X_\tau$ extends to a regular morphism $X \times X \to X$, with $X_\tau \cong \GG$. Since multiplication in $\GG$ is associative and has a neutral element, the same holds for $X$. Therefore, $X$ is an affine monoid with group of invertible elements $\GG$.

\subsection{On the existence of root monoids}
In this subsection we show that using the construction from Theorem~\ref{thm:mon} one can build an affine monoid with any prescribed group of invertible elements $\GG$. 	
\begin{proposition}\label{prop:exist}
	For any group $\GG = \G_a^k \leftthreetimes T$ and any cone $\sigma$ with a $k$-dimensional regular face $\tau$, there exists a monoid structure on $X_\sigma$ whose group of invertible elements is isomorphic to $\GG$. 
\end{proposition}

\begin{proof}
We prove that for any group $\GG$ one can choose Demazure roots in Theorem~\ref{thm:mon} in such a way that the introduced monoid structure has a group of invertible elements isomorphic to~$\GG$. In subsections~\ref{sec:conhom} and~\ref{1sec:orb} we showed that all characters have the form $\chi_r = \chi^{a_r - b_r}$ with $a_r - b_r \in \tau^\perp$. At the same time, $a_r = -e_1^{(r)}$, $b_r = -e_2^{(r)}$ in accordance with the proof of Theorem~\ref{thm:mon}. Thus, to complete the proof it remains to prove the following lemma.
\begin{lemma}
	For any set $c_1, \ldots, c_k \in \tau^\perp \cap M$, where $\tau$ is a regular face of the cone $\sigma$, one can choose a set of Demazure roots $\{ e_1^{(r)}, e_2^{(r)} \}$ compatible with $\tau$ such that for all $r = 1, \ldots, k$ the equality $e_1^{(r)} - e_2^{(r)} = c_r$ holds.
\end{lemma}
\begin{proof}
	Denote by $p_1, \ldots, p_k$ the primitive vectors on the rays $\rho_1, \ldots, \rho_k$ of the cone~$\tau$. Fix $r$. Note that there exists $u_r \in M$ such that 
	\begin{equation}\label{cond}
		\langle p_r, u_r \rangle  = -1 \text{ and } \langle p_s, u_r \rangle  = 0 \text{ for all } s \neq r. 
	\end{equation}
	We will show that one can choose two Demazure roots $e_1^{(r)}, e_2^{(r)}$ satisfying condition~\eqref{cond} in such a way that their difference equals $c_r \in \tau^\perp$. Recall that if $\gamma$ is a face of the cone $\sigma$, then $\gamma^* = \gamma^\perp \cap \sigma^\vee$ is a face of the dual cone $\sigma^\vee$, with $\dim \gamma + \dim \gamma^* = n = \dim \sigma$, see~\cite[Proposition 1.2.10]{F}. Take $v \in M$ from the relative interior of the face $\tau^* = \tau^\perp \cap \sigma^\vee$. Then $\langle p_s, v \rangle = 0$ for all $s = 1, \ldots, k$, and $\langle q, v \rangle > 0$ for all other primitive vectors $q$ on the rays of the cone $\sigma$. Hence, for some $N \in \ZZ_{\geq 0}$ the element $u_r + Nv$ is a Demazure root of the cone $\sigma$ and satisfies condition~\eqref{cond}. The same holds for any lattice point in the cone $u_r + Nv + \tau^*$. Therefore, within this cone one can choose two Demazure roots whose difference equals any prescribed vector $c_r \in \tau^\perp \cap M$. 
\end{proof}
Proposition~\ref{prop:exist} is proved.
\end{proof}

\section{Active monoids}\label{sec:actmon}
\begin{definition}
A semidirect product $G = U \leftthreetimes T$ is called \textit{active} if $\dim T + \dim \Im \psi = \dim G$, where $\psi\colon T \to \Aut U$ is the homomorphism defining the semidirect product.
\end{definition}

The notion of an active semidirect product was introduced in~\cite{YZ}. Recall that by Lemma~\ref{lmYZ} one can introduce coordinates on $U$ such that $$\psi(t) = \mathrm{diag}(\chi_1(t), \ldots, \chi_k(t)) \in \GL(U),$$ where $\chi_1(t), \ldots, \chi_k(t)$ are characters of the torus $T$. The group $G$ is active if and only if the characters $\chi_1(t), \ldots, \chi_k(t)$ are linearly independent.

\begin{remark}
If $G = U \leftthreetimes T$ is active, then $\dim U \leq \dim T$.
\end{remark}

The main goal of this paper is the classification of affine monoids with active groups of invertible elements in terms of the comultiplication on the algebra of regular functions. In~\cite{YZ} it was shown that such monoids are toric varieties. Consider the map $\theta\colon T \times T \rightarrow \Aut(G)$ defined as follows: an element $(t_1, t_2) \in T \times T$ acts on an element $g \in G$ by $t_1 g t_2^{-1}$. Denote $\TT = \Im \theta$. Since the torus $\TT$ acts on $G$ by left and right multiplications, and $X$ is an affine monoid with group of invertible elements $G$, the action of the torus $\TT$ on $G$ extends to an action of $\TT$ on $X$, see~\cite[Theorem 1]{Vin} for the case of reductive monoids in characteristic~$0$ and~\cite[Proposition 1]{Ri1} for the general case.

\begin{proposition}\label{pr:toric} \cite[Proposition 2]{YZ} 
Let $X$ be an affine monoid with active group of invertible elements $G$. Then $X$ is a toric variety with respect to the action of the torus $\TT$, and the group of invertible elements $G$ is invariant under the action of $\TT$.
\end{proposition}

We now show that if the semidirect product is active, then the unipotent radical is commutative.

\begin{lemma}\label{thm:commute}
Let the semidirect product $U \leftthreetimes T$ be active. Then the group $U$ is commutative, and hence $U \cong \G_a^k$.
\end{lemma}

\begin{proof}
Let the structure of the semidirect product ${G = U \leftthreetimes T}$ be defined by a homomorphism $\psi\colon T \rightarrow \Aut(U)$, where $t \mapsto t u t^{-1} = \phi_u$, which in turn induces a map $\mathrm{Ad}\colon T \rightarrow \Aut(\mathfrak{u}) \subseteq \GL(\mathfrak{u})$, $t \mapsto d \phi_u$, where $\mathfrak{u} = \mathrm{T}_e U$ is the tangent Lie algebra.

Choose a basis $e_1, \ldots, e_k$ in the algebra $\mathfrak{u}$ in such a way that the torus acts diagonally: $\Ad(t) e_i = \chi_i(t) e_i$. Note that since all characters are linearly independent, the only eigenvectors are those proportional to the basis elements. To prove commutativity of the group $U$, it suffices to check commutativity of the Lie algebra $\mathfrak{u}$, i.e. it suffices to show that $[e_1, e_2] = 0$. Observe that
\[\Ad(t) [e_1, e_2] = [\Ad(t)e_1, \Ad(t) e_2] = [\chi_1(t)e_1, \chi_2(t)e_2] = \chi_1(t) \chi_2(t) [e_1, e_2] = (\chi_1 + \chi_2)(t) [e_1, e_2].\]
Thus $[e_1, e_2]$ is an eigenvector. But since all characters are linearly independent, the character $\chi_1 + \chi_2$ cannot be zero or coincide with one of $\chi_1, \ldots, \chi_k$. Therefore, $[e_1, e_2] = 0$.
\end{proof}
	
Let us formulate one of the main results of this work.

\begin{theorem}\label{thm:act}
	Any affine algebraic monoid with an active group of invertible elements is isomorphic to a root monoid corresponding to some cone~$\sigma$, its $k$-dimensional regular face~$\tau$, and a set of Demazure roots $\{e_1^{(r)}, e_2^{(r)} \mid r = 1, \ldots, k\}$ compatible with $\tau$. Moreover, the group of invertible elements of the root monoid is active if and only if the differences $e_2^{(r)} - e_1^{(r)}$ are linearly independent.
\end{theorem}

\begin{proof}
	Let $X$ be an affine variety endowed with a monoid structure, and suppose that the group of invertible elements $G = U \leftthreetimes T$ is active. Then $U$ is commutative by Lemma~\ref{thm:commute}, and $G \cong \GG$ according to the notation of Section~\ref{aff}. By Proposition~\ref{pr:toric}, $X$ is an affine toric variety corresponding to a cone $\sigma$, and $\GG$ is also a toric variety $X_\tau$ with respect to the action of the same torus $\TT$, see subsection~\ref{sec:conhom}. Since the embedding $\GG = X_\tau \hookrightarrow X = X_\sigma$ restricts to the identity map on the open torus orbit, the embedding $X_\tau \hookrightarrow X_\sigma$ is an open embedding of toric varieties corresponding to the embedding $S_\sigma \hookrightarrow S_\tau$. Hence, $\tau$ is a face of the cone $\sigma$. 
	
	Note that since the group of invertible elements $\GG$ is an open subset in $X$, the algebra of regular functions $\KK[X]$ is a subalgebra of $\KK[\GG]$. Therefore, the comultiplication on $\KK[X]$ is the restriction of the comultiplication ${\KK[\GG] \rightarrow \KK[\GG] \otimes \KK[\GG]}$ corresponding to multiplication in the group $\GG$ (given by formula~\eqref{1l:comu}).
	
	Using the notation of formula~\eqref{1l:comu}, consider $e_1^{(r)} = -a_r$ and $e_2^{(r)} = -b_r$ for $r = 1, \ldots, k$. We will show that $e_1^{(r)}$ and $e_2^{(r)}$ are Demazure roots of the cone~$\sigma$.
	
	Consider two $\G_a$-actions corresponding to the subgroup along the $r$-th coordinate \linebreak[4] ${\G_a \subset \G_a^k \leftthreetimes T \cong \GG}$:
    \begin{equation}
   		\begin{gathered}
    		\G_a \times \GG \rightarrow \GG, \; (\alpha_r',\, \ol{\alpha}\cdot t) \mapsto (\ol{\alpha_r'} + \ol{\alpha})\cdot t, \text{ where } \ol{\alpha_r'} = (0, \ldots, \alpha_r', \ldots, 0); \\
    		\GG \times \G_a \rightarrow \GG, \;   (t \cdot\ol{\beta},\, \beta_r') \mapsto t \cdot (\ol{\beta} + \ol{\beta_r'}) , 
    		\text{ where } \ol{\beta_r'} = (0, \ldots, \beta_r', \ldots, 0).
   		\end{gathered}
   	\end{equation}

	Take the corresponding LNDs on the algebra \[\KK[\GG] = \KK[T][\Alp{1}]\ldots[\Alp{k}] = \KK[T][\Bet{1}]\ldots[\Bet{k}],\] namely, differentiate the dual $\G_a$-actions on $\KK[\GG]$ at $\alpha_r' = 0$ and $\beta_r' = 0$, respectively:
	\begin{equation*}
    	\delta_{r} = \frac{d}{d \Alp{r}},  \partial_{r} = \frac{d}{d \Bet{r}}.
    \end{equation*}

	Since left and right multiplication on $\GG$ extends to multiplication on the whole $X_\sigma$, the algebra $\KK[X_\sigma]$ is invariant under the actions of $\GG$ by left and right multiplication, and $\KK[X_\sigma]$ is invariant under the derivations $\delta_{r}, \partial_{r}$. Moreover, $\delta_{r}, \partial_{r}$ are homogeneous LNDs with respect to the grading $\KK[X_\sigma] = \bigoplus_{u \in S_\sigma} \KK \chi^u$ and of degrees $e_1^{(r)}, e_2^{(r)}$, since $\Alp{r}^{-1} = \chi^{e_1^{(r)}}, \Bet{r}^{-1} = \chi^{e_2^{(r)}}$. Hence, $e_1^{(r)}$ and $e_2^{(r)}$ are Demazure roots. Since ${\langle p_s, e_1^{(r)} \rangle = \langle p_s, e_2^{(r)} \rangle = -\delta_{rs}}$, these Demazure roots are compatible with the face $\tau$.
	
	Substituting $\Alp{r}^{-1} = \chi^{e_1^{(r)}}, \Bet{r}^{-1} = \chi^{e_2^{(r)}}$ into formula~\eqref{1l:comu}, we obtain 
	\begin{equation}\label{form}
		\chi^u \mapsto \chi^u \otimes  \chi^u  \prod_{r = 1}^{k} (1 \otimes \chi^{e_1^{(r)}}  + \chi^{e_2^{(r)}} \otimes 1)^{\langle p_r, u\rangle},
	\end{equation}
	which coincides with the formula for comultiplication from Theorem~\ref{thm:mon}.
\end{proof}

\begin{remark}
	In accordance with Theorem~\ref{thm:act} and Proposition~\ref{prop:exist}, for any active group $G$ with a $k$-dimensional unipotent radical and any cone $\sigma$ with a $k$-dimensional regular face, there exists a monoid structure on $X_\sigma$ with group of invertible elements isomorphic to~$G$.
\end{remark}

\section{Examples}\label{sec:ex}
\subsection{Affine space}

In this section we describe all root monoid structures on the affine space $\mathbb{A}^n$.

Note that affine space is an affine toric variety $\mathbb{A}^n \cong X_\sigma$, where the primitive vectors on the rays of the cone $\sigma$ form a basis of the lattice. Consider a $k$-dimensional face $\tau \subset \sigma$. Observe that all such faces are equivalent up to isomorphism, hence we may assume $\tau = \cone(p_1, \ldots, p_k)$, where $p_r = (0, \ldots, 0, \overset{r}{1}, 0, \ldots, 0)$. For all $r = 1, \ldots, k$ choose pairs of Demazure roots corresponding to the ray $\rho_r$ as follows:
\begin{gather*}
e_1^{(r)} =  (\underbrace{0, \ldots, 0, \overset{r}{-1}, 0, \ldots, 0}_{k}, a_{k+1}^{(r)}, \ldots, a_n^{(r)}),
\\
e_2^{(r)} =  (\underbrace{0, \ldots, 0, \overset{r}{-1}, 0, \ldots, 0}_{k}, b_{k+1}^{(r)}, \ldots, b_n^{(r)}),
\end{gather*}
where $a_{k+1}^{(r)}, \ldots, a_n^{(r)}, b_{k+1}^{(r)}, \ldots, b_n^{(r)} \in \ZZ_{\geq 0}$. Note that all sets of Demazure roots compatible with~$\tau$ are of this form. Now apply formula~\eqref{thm:finalcom} to the basis vectors $u = u_i$ corresponding to the coordinates $x_i$ on the affine space $\mathbb{A}^n$, where $i = 1, \ldots, n$. Denote by $\ol{a}_r$ the tuple $(a_{k+1}^{(r)}, \ldots, a_n^{(r)})$, and by $\ol{b}_r$ the tuple $(b_{k+1}^{(r)}, \ldots, b_n^{(r)})$, respectively. In addition, denote the monomial $x^{\ol{b}_r} = x_{k+1}^{b_{k+1}^{(r)}}\ldots x_{n}^{b_{n}^{(r)}}$, and similarly, $y^{\ol{a}_r} = y_{k+1}^{a_{k+1}^{(r)}}\ldots y_{n}^{a_{n}^{(r)}}$.
	
For $i = k+1, \ldots, n$ we obtain $\chi^{u_i}(x*y) = \chi^{u_i}(x) \chi^{u_i}(y) = x_iy_i$, since in this case $\langle p_r, u \rangle = 0$. If $i = 1, \ldots, k$, we obtain: $$\chi^{u_i}(x*y) = \chi^{u_i}(x) \chi^{u_i}(y) (\chi^{e_1^{(i)}}(y) + \chi^{e_2^{(i)}}(x) ) = x_iy_i(y_i^{-1}y^{\ol{a}_i} + x_i^{-1}x^{\ol{b}_i}) = x_iy^{\ol{a}_i} + y_ix^{\ol{b}_i}.$$
That is,
$$x*y = (x_1y^{\ol{a}_1} + y_1x^{\ol{b}_1}, \ldots, x_ky^{\ol{a}_k} + y_kx^{\ol{b}_k}, x_{k+1} y_{k+1}, \ldots, x_ny_n).$$

The monoid constructed above is active if and only if the differences ${\ol{b}_1 - \ol{a}_1}, \ldots, \ol{b}_k - \ol{a}_k$ are linearly independent. For example, the multiplication
\[(x_1, x_2, x_3, x_4)*(y_1, y_2, y_3, y_4) = (x_1 + y_1x_3x_4^2, x_2y_3 + y_2x_3^3x_4^4, x_3y_3, x_4y_4)\]
defines a non-active root monoid on $\AA^4$, whereas the multiplication
\[(x_1, x_2, x_3, x_4)*(y_1, y_2, y_3, y_4) = (x_1 + y_1x_3x_4, x_2y_3 + y_2x_3^3x_4^4, x_3y_3, x_4y_4)\]
determines an active (root) monoid on $\AA^4$.

\subsection{Cylinder over a quadratic cone} \label{example2}
Consider the cone $\sigma$ in the $4$-dimensional lattice $N$ with primitive vectors
\[p_1 = (1, 0, 0, 0), \, p_2 = (0, 1, 0, 0), \, p_3 = (0, 0, 1, 0), \, p_4 = (0, 1, 0, 1), \, p_5 = (1, 0, 0, 1).\] The dual cone $\sigma^\vee \subset M$ is defined by the vectors
\[q_1 = (1, 0, 0, 0), \, q_2 = (0, 1, 0, 0), \, q_3 = (0, 0, 1, 0), \, q_4 = (0, 0, 0, 1), \, q_5 = (1, 1, 0, -1),\] which also generate the semigroup $S_\sigma = \sigma^\vee \cap M$. Consider the functions $x_i = \chi^{q_i}$. The algebra of regular functions $\KK[X]$ of the corresponding toric variety $X$ is isomorphic to $\KK[x_1, x_2, x_3, x_4, x_5]$ with the relation $x_1x_2 = x_4x_5$. Hence, $$X = \{ x_1x_2 = x_4x_5 \} \subseteq \mathbb{A}^5,$$ that is, $X$ is the direct product of the quadratic cone and the affine line.

Consider the two-dimensional regular cone $\tau = \cone(p_1, p_2)$. Choose Demazure roots compatible with the cone $\tau$ (two for each primitive vector $p_1, p_2$):
\begin{equation*}
	\begin{gathered}
		e_1^{(1)} = (-1, 0, a_1, b_1), \text{ where } a_1 \geq 0, b_1 > 0,\\
		e_2^{(1)} = (-1, 0, a_2, b_2), \text{ where } a_2 \geq 0, b_2 > 0,\\
		e_1^{(2)} = (0, -1, c_1, d_1), \text{ where } c_1 \geq 0, d_1 > 0,\\
		e_2^{(2)} = (0, -1, c_2, d_2), \text{ where } c_2 \geq 0, d_2 > 0 .
	\end{gathered}
\end{equation*}
With this choice of Demazure roots, formula~\eqref{thm:finalcom} takes the form: $$\chi^u \mapsto \chi^u \otimes  \chi^u  (1 \otimes \chi^{e_1^{(1)}}  + \chi^{e_2^{(1)}} \otimes 1)^{\langle p_1, u\rangle} (1 \otimes \chi^{e_1^{(2)}}  + \chi^{e_2^{(2)}} \otimes 1)^{\langle p_2, u\rangle}. $$ 
The multiplication on $X$ is given by the formula
\begin{multline}\label{eq_mult_example2}
	x*y = \bigl(x_1 y_3^{a_1}y_4^{b_1} + y_1 x_3^{a_2}x_4^{b_2}, \;\; x_2 y_3^{c_1}y_4^{d_1} + y_2 x_3^{c_2}x_4^{d_2}, 
	\\ x_3 y_3, \;\; x_4 y_4, \;\;
	x_5y_3^{a_1 + c_1}y_4^{b_1 + d_1 - 1} 
	+ y_5x_3^{a_2 + c_2}x_4^{b_2 + d_2 - 1} +
	\\
	+ x_1 x_3^{c_2}x_4^{d_2 - 1}y_2 y_3^{a_1}y_4^{b_1 - 1} 
	+ y_1 y_3^{c_1}y_4^{d_1 - 1}x_2 x_3^{a_2}x_4^{b_2 - 1}\bigr).
\end{multline}
The corresponding monoid is active if and only if two vectors 
$(a_1 - a_2, b_1 - b_2)$ and ${(c_1 - c_2, d_1 - d_2)}$ are linearly independent.

\section{Idempotents and torus orbits}\label{sec:idemp} 

Recall that an element $x \in X$ is called an \textit{idempotent} if $x * x = x$. Idempotents in algebraic monoids have been extensively studied in various works, see, for example,~\cite{Br-1, Br-2, Neeb, Re} and the references therein. In particular, it is known that the subvariety of idempotents is smooth and is an orbit of the group of invertible elements under the action by conjugation, see~\cite[Theorem~1.1]{Br-2}. In this section we describe all idempotents in each torus orbit for root monoids. 

\begin{theorem}\label{idem_thm_1}
	In the notation of Theorem~\ref{thm:mon}, let $X_\sigma$ be a root monoid, $\tau$ the corresponding $k$-dimensional regular face of the cone $\sigma$ with primitive vectors $p_1, \ldots, p_k$ on the one-dimensional faces. Let $\gamma$ be a face of the cone $\sigma$. Denote by $E_\gamma$ the set of idempotents lying in the orbit $O_\gamma$. Then:
	\begin{enumerate}
		\item if $\tau \subseteq \gamma$, then $E_\gamma = \{x_\gamma \}$; 
		\item if there exists a ray $p_r \notin \gamma$ such that either corresponding Demazure roots $e_1^{(r)}, e_2^{(r)} \notin \gamma^\perp$, or $e_1^{(r)}, e_2^{(r)} \in \gamma^\perp$, then $E_\gamma = \varnothing$;
		\item if for every such ray $p_r \notin \gamma$ exactly one of the corresponding Demazure roots $e_1^{(r)}, e_2^{(r)}$ lies in $\gamma^\perp$, then \[E_\gamma = O_\gamma \cap \{\chi^u = 1\;\; \forall u \in \cone(\tau, \gamma)^\perp \cap S_\sigma\}.\]
	\end{enumerate}
\end{theorem}

\begin{proof}
	Note that $x \in X$ is an idempotent if and only if $\chi^u(x * x) = \chi^u(x)$ for every $u \in S_\sigma$.
	
	\bigskip
	
	(1) Suppose $\tau \subseteq \gamma$. Take $u \in \gamma^\perp \subseteq \tau^\perp$. Then $\langle p_r, u\rangle = 0$ for all $r = 1, \ldots, k$. By formula~\eqref{thm:finalcom} we have $\chi^u(x*x) = \chi^u(x)\cdot\chi^u(x)$. Hence $x$ can be an idempotent only if $\chi^u(x) = 1$ or $\chi^u(x) = 0$. Since $x \in O_\gamma, u \in \gamma^\perp$, by formula~\eqref{tor:fun_orb2} we know $\chi^u(x) \neq 0$, so $\chi^u(x) = 1$ and $x = x_\gamma$. Note that therefore there is at most one idempotent in the orbit $O_\gamma$. Let us check that $x_\gamma$ is indeed an idempotent. Now let $u \in \tau^\perp, u \notin \gamma^\perp$. Then $\chi^u(x_\gamma*x_\gamma) = \chi^u(x_\gamma)\cdot \chi^u(x_\gamma) = 0 \cdot 0 = 0 = \chi^u(x_\gamma)$. It remains to consider the case $u \notin \tau^\perp$. Recall that by Theorem~\ref{thm:mon}
	\begin{equation}\label{sum_form}
		\chi^u (x * y) = \sum \limits_{\ol{i} + \ol{j} = \langle \ol{p}, u \rangle } { \langle \ol{p}, u \rangle \choose \ol{i} } \chi^{u + \ol{i}\cdot \ol{e_2}}(x) \cdot  \chi^{u + \ol{j}\cdot \ol{e_1}} (y).
	\end{equation}
	
	Let us show that each summand in the right-hand side of~\eqref{sum_form} is equal to $0$. Recall that we denote $u + \ol{i}\cdot \ol{e_2} = u + i_1 e_2^{(1)} + \ldots + i_k e_2^{(k)}$. Then $\langle p_s, u + \ol{i}\cdot \ol{e_2} \rangle = \langle p_s, u\rangle - i_s$ for any $s = 1, \ldots, k$. Similarly $\langle p_s, u + \ol{j}\cdot \ol{e_1} \rangle = \langle p_s, u\rangle - j_s$. Since $\ol{i} + \ol{j} = \langle \ol{p}, u \rangle \neq \ol{0}$, at least one of $u + \ol{i}\cdot \ol{e_2}$ or $u + \ol{j}\cdot \ol{e_1}$ does not lie in $\tau^\perp$, and hence not in $\gamma^\perp$. Therefore $\chi^{u + \ol{i}\cdot \ol{e_2}}(x_\gamma) \cdot  \chi^{u + \ol{j}\cdot \ol{e_1}} (x_\gamma) = 0$. Then $\chi^u(x_\gamma*x_\gamma) = 0 = \chi^u(x_\gamma)$. Thus $x_\gamma$ is the unique idempotent in $O_\gamma$.
	
	\bigskip
	
	Let us divide case (2) into two subcases:\\
	(2.1) $p_r \notin \gamma$ and $e_1^{(r)}, e_2^{(r)} \notin \gamma^\perp$ for some $r$;\\
	(2.2) $p_r \notin \gamma$ and $e_1^{(r)}, e_2^{(r)} \in \gamma^\perp$ for some $r$, and there exists no $j$ such that $p_j \notin \gamma$ but ${e_1^{(j)}, e_2^{(j)} \notin \gamma^\perp}$.
	
	\medskip
	
	(2.1) Suppose $p_r \notin \gamma$ and $e_1^{(r)}, e_2^{(r)} \notin \gamma^\perp$. Then there exists a ray of the cone $\gamma$ with primitive vector $q$ such that $\langle q, e_1^{(r)} \rangle > 0$. Indeed, since $e_1^{(r)}\notin \gamma^\perp$, there exists $q \in \gamma$ with nonzero pairing. Since $e_1^{(r)}$ is a Demazure root, we have $\langle p_r, e_1^{(r)} \rangle = -1$, and with all other primitive vectors on rays of $\sigma$ the pairing is nonnegative, in particular $\langle q, e_1^{(r)} \rangle \geq 0$.  

	Next note that since $p_r$ is not a ray of $\gamma$, we have $\gamma^\perp \setminus p_r^\perp \neq \varnothing$. Take $u \in \gamma^\perp \setminus p_r^\perp$. Note that $i_r + j_r = \langle p_r, u \rangle \neq 0$, since $u \notin p_r^\perp$. Without loss of generality, assume $j_r \neq 0$. Then \[\langle q, u + \ol{j} \cdot \ol{e_1} \rangle = \langle q, u \rangle + j_1 \langle q, e_1^{(r)} \rangle + \ldots + j_k\langle q, e_1^{(r)} \rangle > 0.\] Indeed, $\langle q, u \rangle = 0$ since $q \in \gamma$, $u \in \gamma^\perp$, $j_r \langle q, e_1^{(r)} \rangle > 0$, and all other terms are nonnegative. Hence $u + \ol{j} \cdot \ol{e_1} \notin \gamma^\perp$. Thus for $x \in O_\gamma$, we get $\chi^{u + \ol{j} \cdot \ol{e_1}}(x)\cdot \chi^{u + \ol{i} \cdot \ol{e_2}}(x) = 0$. Then by~\eqref{sum_form}, $\chi^u(x*x) = 0$. On the other hand, $u \in \gamma^\perp$, whence $\chi^u(x) \neq 0$. Hence there are no idempotents in the orbit $O_\gamma$. 
	
	\bigskip
	
	For cases (2.2) and (3) we prove auxiliary lemmas. Recall that $\tau = \cone(p_1, \ldots, p_k)$; without loss of generality assume ${p_1, \ldots, p_s \notin \gamma}$, ${p_{s+1}, \ldots, p_k \in \gamma}$. 
	
	\begin{lemma}\label{lm:cone}
		Suppose that for each ray $p_r \notin \gamma$, where $r = 1, \ldots, s$, at least one of the corresponding Demazure roots $e_1^{(r)}$ or $e_2^{(r)}$ lies in $\gamma^\perp$. Then $\cone(p_1, \ldots, p_s, \gamma)$ is a face of the cone~$\sigma$.
	\end{lemma}
	
	\begin{proof}
	Denote by $e^{(r)}$ the Demazure root lying in $\gamma^\perp$, and by $\gamma_r$ the cone $\cone(p_1, \ldots, p_r, \gamma)$. We will prove by induction on $r$ that $\gamma_r$ is a face of $\sigma$. For $r = 0$ this is clear. Let $\gamma_r$ be a face of~$\sigma$, hence there exists $m_r \in M$ such that $\langle x, m_r\rangle = 0$ for all $x \in \gamma_r$, and $\langle x, m_r\rangle > 0$ for all other $y \in \sigma$. Take the Demazure root $e^{(r+1)} \in \gamma^\perp$ corresponding to the ray $p_{r+1}$. We will look for a similar vector for $\gamma_{r+1}$ of the form
	$$m_{r+1}:=m_r + \alpha e^{(r+1)}. $$ 
	Since the Demazure roots are compatible with the cone $\tau$, we have $\langle p_i,  e^{(r+1)} \rangle = -\delta_{i, r+1}$, where $\delta$ is the Kronecker symbol. Let $x \in \gamma_r$, that is, $x = x' + a_1 p_1 + \ldots + a_r p_r$, where ${x' \in \gamma}$. Since $e^{(r + 1)} \in \gamma^\perp$, we get $\langle x , m_{r+1} \rangle = \langle x , m_r \rangle + \langle x , \alpha e^{(r+1)} \rangle = 0$ for any $\alpha \in \ZZ$. For $x \notin \gamma_r$ we have $\langle x , m_{r+1} \rangle = \langle x , m_r \rangle  + \langle x , \alpha e^{(r+1)} \rangle > 0$, since the first term is strictly positive and the second is nonnegative. Next, note that $p_{r+1} \notin \gamma_r$, since $e^{(r+1)} \in \gamma_r^\perp$, but $e^{(r+1)} \notin p_{r+1}^\perp$. Therefore, $\langle p_{r+1} , m_r  \rangle > 0 $, and since $\langle p_{r+1} , e^{(r+1)} \rangle = -1$, we can choose $\alpha > 0$ such that $\langle p_{r+1} ,m_{r+1}\rangle = 0$. Now $\langle x, m_{r+1} \rangle = 0$ for all $x \in \cone(p_{r+1}, \gamma_r) = \gamma_{r+1}$, and $\langle x, m_{r+1} \rangle > 0$ for all other $x \in \sigma$. Hence $\gamma_{r+1}$ is also a face of the cone $\sigma$.
	\end{proof}

\begin{lemma}\label{lm:perpcone}
Suppose that for each of the rays $p_r \notin \gamma$, where $r = 1, \ldots, s$, at least one of the corresponding Demazure roots $e_1^{(r)}$ or $e_2^{(r)}$ lies in $\gamma^\perp$. Then for any $1 \leq r \leq s$ there exists $u \in \gamma^\perp \cap \sigma^\vee$ such that $\langle p_r, u \rangle = 1$, and $\langle p_j, u \rangle = 0$ for all $j \neq r,\; 1 \leq j \leq k$.
\end{lemma}

\begin{proof}
Let $e^{(r)} \in \gamma^\perp$ be the Demazure root corresponding to $p_r$. Note $\gamma^\perp \cap \{ \langle p_r, \cdot \rangle = 1\} \ni -e^{(r)}$. Take $w$ from the relative interior of the face $\gamma_s^\perp \cap \sigma^\vee$, where $\gamma_s = \cone(p_1, \ldots, p_s, \gamma) = \cone(\tau, \gamma)$ is a face of the cone $\sigma$ by Lemma~\ref{lm:cone}. Then $\langle p_j, w \rangle = 0$ and $\langle q, w \rangle > 0$ for the remaining primitive vectors $q$ on the rays of the cone $\sigma$ (not contained in $\gamma$). We will find $u$ in the form $u = -e^{(r)} + \alpha w$, where $\alpha \in \ZZ_{\geq 0}$. Note that $u \in \gamma^\perp$, and $\langle p_j, u \rangle = 0$ for $j \neq r$, $\langle p_r, u \rangle = 1$. Moreover, one can choose $\alpha$ in such a way that $\langle q, u \rangle \geq 0$ for the remaining primitive vectors $q$ on the rays of the cone $\sigma$, hence $u \in \sigma^\vee$.
\end{proof}

(2.2) Let $p_r \notin \gamma$, and $e_1^{(r)}, e_2^{(r)} \in \gamma^\perp$. Note that if there is no ray $p_j \notin \gamma$ for which both corresponding Demazure roots $e_1^{(j)}, e_2^{(j)} \notin \gamma^\perp$, then we may assume that for each of the rays $p_1, \ldots, p_s$, at least one of the corresponding Demazure roots $e_1^{(r)}$ or $e_2^{(r)}$ lies in $\gamma^\perp$. Therefore, we may apply Lemma~\ref{lm:perpcone}: there exists $u \in \gamma^\perp$ such that $\langle p_r, u \rangle = 1$, and $\langle p_j, u \rangle = 0$ for all $j \neq r$. Note that then $u + e_l^{(r)} \in \gamma^\perp \cap \tau^\perp$, $l = 1, 2$. Consider $x \in O_\gamma$, and suppose that $x$ is an idempotent. Let us prove that $\chi^{u + e_l^{(r)}}(x) = 1$. Indeed, by formula~\eqref{thm:finalcom} we have $\chi^{u + e_l^{(r)}}(x*x) = \chi^{u + e_l^{(r)}}(x)\cdot\chi^{u + e_l^{(r)}}(x)$. Therefore, $x$ can be an idempotent only if ${\chi^{u + e_l^{(r)}}(x) = 1}$ or $\chi^{u + e_l^{(r)}}(x) = 0$. Since $x \in O_\gamma,\; u + e_l^{(r)} \in \gamma^\perp$, it follows that $\chi^{u + e_l^{(r)}}(x) \neq 0$, hence ${\chi^{u + e_l^{(r)}}(x) = 1}$. On the other hand, since $\langle p_r, u \rangle = 1$, and $\langle p_j, u \rangle = 0$ for $j \neq r$, by formula~\eqref{thm:finalcom} we obtain $\chi^u(x*x) = \chi^u(x)\chi^{u + e_1^{(r)}}(x) + \chi^{u + e_2^{(r)}}\chi^u(x) = \chi^u(x) + \chi^u(x)$. That is, $\chi^u(x*x) = \chi^u(x)$ only when $\chi^u(x) = 0$, which contradicts $x \in O_\gamma$. Hence, $E_\gamma = \varnothing$.

\bigskip

(3) Now consider the case when for each $p_1, \ldots, p_s \notin \gamma$ exactly one of the Demazure roots $e_1^{(r)}$ or $e_2^{(r)}$ lies in $\gamma^\perp$. Note that in the expansion of $\chi^u(x*x)$ by formula~\eqref{thm:finalcom} the result does not change if we swap $e_1^{(r)}$ and $e_2^{(r)}$. Therefore, we may assume that $e_2^{(r)} \in \gamma^\perp$ and $e_1^{(r)} \notin \gamma^\perp$ for all $r = 1, \ldots, s$. Note that then for each root $e_1^{(r)}$, where $r = 1, \ldots, s$, there exists a ray of the cone~$\gamma$ and a primitive vector $q_r$ on this ray such that $\langle q_r, e_1^{(r)} \rangle \neq 0$. Since $p_r \notin \gamma$ while $q_r \in \gamma$, we have $p_r \neq q_r$, hence $\langle q_r, e_1^{(r)} \rangle > 0$. Moreover, since the set of Demazure roots is compatible with the cone $\tau$, i.e., $\langle p_i, e_1^{(r)} \rangle = -\delta_{ri}$ for any $i = 1, \ldots, k$, the vector $q_r$ does not coincide with any of $p_1, \ldots, p_k$. Therefore, the pairing of $q_r$ with all Demazure roots $\{e_1^{(r)}, e_2^{(r)} \mid r = 1,\ldots,k\}$ is nonnegative.

Consider $u \in \sigma^\vee \setminus \gamma^\perp$. We will check that $\chi^u(x*x) = 0$ for all $x \in O_\gamma$. Since $u \notin \gamma^\perp$, we obtain $\chi^u(x) = 0 = \chi^u(x*x)$ for all $x \in O_\gamma$. Denote by $\tau_0$ the face $\cone(p_{s+1}, \ldots, p_k) \subseteq \gamma$ of the cone $\tau$.

Let $u \in \sigma^\vee \setminus \gamma^\perp$ and $u \in \tau_0^\perp$. Choose the primitive vector $q$ on a ray of the cone $\gamma$ such that $\langle q, u \rangle > 0$. Note that then $q \notin \tau_0$, and hence $q$ has a nonnegative pairing with all Demazure roots $\{e_1^{(r)}, e_2^{(r)} \mid r = s+1,\ldots,k\}$. It follows that $\langle q + c_1 q_1 + \ldots + c_s q_s, u \rangle > 0$ for any choice of $c_r \in \ZZ_{\geq 0}$. Choose coefficients $c_r$ in such a way that $\langle q + c_1 q_1 + \ldots + c_s q_s, e_1^{(r)} \rangle > 0$ for all $r = 1, \ldots, s$; this is possible since $\langle q_r, e_1^{(r)} \rangle > 0$. Set $q' := q + c_1 q_1 + \ldots + c_s q_s \in \gamma$. Then $\langle q', u + j_1 e_1^{(1)} + \ldots + j_k e_1^{(k)} \rangle > 0$. Indeed, $\langle q', u \rangle > 0$, and the other terms are nonnegative. Thus, $u + j_1 e_1^{(1)} + \ldots + j_k e_1^{(k)} = u + \ol{j} \ol{e_1} \notin \gamma^\perp$, and by formula~\eqref{thm:finalcom} we obtain $\chi^u(x*x) = 0$ for all $x \in O_\gamma$.

Now let $u \in \sigma^\vee \setminus \gamma^\perp$ and $u \notin \tau_0^\perp$. Let us proceed with the reasoning analogous to that in case~(1). Recall that $\langle p_r, u + \ol{i}\cdot \ol{e_2} \rangle = \langle p_r, u \rangle - i_r = j_r$ and $\langle p_r, u + \ol{j}\cdot \ol{e_1} \rangle = \langle p_r, u \rangle - j_r = i_r$ for any $r = 1, \ldots, k$. Note that since $u \notin \tau_0^\perp$, it follows that $i_r + j_r = \langle p_r, u \rangle \neq 0$ for some $r \geq s+1$, hence $u + \ol{i}\cdot \ol{e_2}$ or $u + \ol{j}\cdot \ol{e_1}$ does not lie in $\tau_0^\perp$, and therefore not in $\gamma^\perp$, hence $\chi^{u + \ol{i}\cdot \ol{e_2}}(x_\gamma) \cdot \chi^{u + \ol{j}\cdot \ol{e_1}}(x_\gamma) = 0$. By formula~\eqref{thm:finalcom} we obtain $\chi^u(x*x) = 0$ for all $x \in O_\gamma$.

Now consider $u \in \gamma^\perp$. Then $\langle p_{s+1}, u \rangle = \ldots = \langle p_{k}, u \rangle = 0$. Then $j_{s+1} = \ldots = j_k = 0$ in formula~\eqref{thm:finalcom}. Let us show that $u + \ol{j} \ol{e_1} \notin \gamma^\perp$ for $\ol{j} \neq \ol{0}$. Suppose $j_r > 0$, where $r \leq s$. Then $\langle q_r, u + \ol{j} \ol{e_1} \rangle > 0$. Indeed, $j_r \langle q_r, e_1^{(r)} \rangle > 0$, while the other terms are nonnegative. Also note that $u + \ol{i} \ol{e_2} \in \gamma^\perp$. After expanding $\chi^u(x*x) = \chi^u(x)$ by formula~\eqref{thm:finalcom}, only one nonzero term remains on the left-hand side, corresponding to $\ol{j} = \ol{0}$. That is, we obtain the equality $\chi^u(x)\chi^{u + \langle \ol{p}, u \rangle \ol{e_2}}(x) = \chi^u(x)$. Since $u \in \gamma^\perp$, we have $\chi^u(x) \neq 0$, and we obtain $$\chi^{u + \langle \ol{p}, u \rangle \ol{e_2}}(x) = 1.$$ Denote $w = u + \langle \ol{p}, u \rangle \ol{e_2} = u + \langle p_1, u \rangle e_2^{(1)} + \ldots + \langle p_s, u \rangle e_2^{(s)}$. Notice that $w \in p_i^\perp$ for any $i = 1, \ldots, s$. Indeed, $\langle p_i, w \rangle = \langle p_i, u \rangle - \langle p_i, u \rangle = 0$. Thus, $w \in p_1^\perp \cap \ldots \cap p_s^\perp \cap \gamma^\perp$. In other words, $w \in \cone(\tau, \gamma)^\perp$. Moreover, choosing $u \in \cone(\tau, \gamma)^\perp$, we obtain as $w$ all possible elements of $\cone(\tau, \gamma)^\perp$, since in this case $w = u$.

Thus, the idempotents in $O_\gamma$ are the points $x$ satisfying $\chi^u(x) = 1$ for all $u \in \cone(\tau, \gamma)^\perp$.
\end{proof}

\section{Idempotents and actions of one-parameter subgroups} \label{sec:idemp_geom}
In the previous section, for each orbit $O_\gamma$ of the torus $\TT$ in the root monoid~$X_\sigma$, we found the set of idempotents $E_\gamma$ in the orbit $O_\gamma$. In this section, we describe the closure of the set $E_\gamma$. 

Let $X_\sigma$ be a root monoid with the group of invertible elements $\GG = \G_a^k \leftthreetimes T$, defined by the regular cone $\tau = \cone(p_1, \ldots, p_{k})$ and the set of Demazure roots $\{e_1^{(r)}, e_2^{(r)} \mid r = 1,\ldots, k\}$. 

For a face $\gamma \subseteq \sigma$ such that for all rays $p_1, \ldots, p_s \notin \gamma$ exactly one of the corresponding Demazure roots $e_1^{(r)}, e_2^{(r)}$ lies in $\gamma^\perp$, we construct a unipotent group $H_\gamma$ as follows. 
For each $r = 1, \ldots, s$, choose from the pair $e_1^{(r)}, e_2^{(r)}$ the Demazure root not lying in $\gamma^\perp$, and denote by~$H_r$ the corresponding one-parameter subgroup in $\Aut(X_\sigma)$ normalized by the torus $\TT$, see~\ref{root}. Put \[H_\gamma:= \langle H_1, \ldots, H_s \rangle \subseteq \Aut(X_\sigma).\] Recall also that the torus $T$ is identified with a subtorus of $\TT$, so that we may consider the action of $T$ on $X_\sigma$. 

\begin{theorem}\label{idem_thm_2}
	In the notation of Theorems~\ref{idem_thm_1} and~\ref{thm:mon}, if $E_\gamma \neq \varnothing$, then
	\begin{enumerate}
		\item $\ol{E_\gamma} = \bigcup_{\{i_1, \ldots, i_l\} \subseteq \{1, \ldots, s\}} E_{\cone(\gamma, p_{i_1}, \ldots, p_{i_l})}$;
		\item $E_\gamma$ is the $T$-orbit of the point $x_\gamma$;
		\item $\ol{E_\gamma}$ is the $H_\gamma$-orbit of the point $x_{\cone(\tau, \gamma)}$.
	\end{enumerate}
\end{theorem}

\subsection{One-parameter subgroups}
Before proving the theorem, we recall some results from~\cite{AKZ2012}. Let $X_\sigma$ be an affine toric variety with acting torus~$\TT$, let $p$ be a primitive vector on a ray of the cone~$\sigma$, and let $e$ be the Demazure root corresponding to~$p$. Denote by $H_e$ the one-parameter subgroup in $\Aut(X_\sigma)$ normalized by the torus~$\TT$ and corresponding to~$e$, see~\ref{root}. 

In \cite[Proposition~2.1]{AKZ2012} it was proved that every orbit of the group $H_e$ in $X_\sigma$ is either a point or intersects exactly two $\TT$-orbits $O_1$, $O_2$, where ${\dim O_1 = \dim O_2 + 1}$ and $O_2 \subseteq \overline{O_1}$. Moreover, in the latter case
\begin{equation}
	\label{Heorbit_eq}
	\begin{aligned}
		O_1 \cap H_ex = R_px,\\
		O_2 \cap H_ex = \{pt\},
	\end{aligned}
\end{equation}
where $x$ is any point in $O_1$, and $R_p \cong \KK^\times$ is the one-parameter subgroup in $\Aut(X_\sigma)$ given by the primitive vector $p \in N$. Such a pair of $\TT$-orbits $(O_1, O_2)$ is called \emph{$H_e$-connected}. 

In~\cite[Lemma 2.2]{AKZ2012} a combinatorial condition on the faces of $\sigma$ corresponding to $H_e$-connected orbits is given, from which it follows that 
	\begin{equation}\label{idempHe_lemeq}
		\text{for a face $\gamma'$ the pair $(O_\gamma, O_{\gamma'})$ is $H_e$-connected } \; \Leftrightarrow \; 
		\left\{\begin{aligned}
			&\text{$\rho$ is not a ray of the cone~$\gamma$},\\
			&\gamma' = \cone(\gamma, \rho),\\
			&e \in \gamma^\perp,
		\end{aligned}\right.
	\end{equation}
see also~(14) in~\cite{YZ}.

In our case of a root monoid, for brevity we denote the one-parameter subgroup corresponding to the ray $p_i$ by $R_i = R_{p_i}$, where $i = 1, \ldots, k$. Clearly,
\[T = R_\tau = \langle R_1, \ldots, R_k\rangle.\]

\subsection{Proof of Theorem~\ref{idem_thm_2}}
We prove several auxiliary statements. We assume $E_\gamma \neq \varnothing$. 

\begin{lemma}\label{invar}
The set of idempotents in $X_\sigma$ is invariant under all $R_i$ with $p_i \in \tau$. In particular, the set of idempotents is invariant under $R_\tau$.	
\end{lemma}

\begin{proof} Take any $x \in E_\gamma$. We check that $R_i(t) x \in E_\gamma$. By Theorem~\ref{idem_thm_1} we need to verify that $\chi^u(R_i(t) x) = 1$ for all ${u \in \cone(\tau, \gamma)^\perp}$, which is equivalent to $R_i(t)\chi^u(x) = 1$. Note that $R_i(t)\chi^u = t^{\langle p_i, u \rangle}\chi^u = \chi^u$ since $\langle p_i, u \rangle = 0$. Moreover, by Theorem~\ref{idem_thm_1} we have $\chi^u (x) = 1$ since $x \in E_\gamma$. The set $E_\gamma$ is $R_\tau$-invariant, because $R_\tau$ is generated by all~$R_i$.	
\end{proof}

\begin{lemma}\label{st:h-idem}
The set of idempotents in $X_\sigma$ is invariant under all $H_i$ with $p_i \in \tau$. In particular, the set of idempotents is invariant under $H_\gamma$.
\end{lemma}

\begin{proof} Let $x \in X_\sigma$ be an idempotent. Recall that $|H_i x| = 1$ or $\ol{R_i x} = R_i x \cup \{x'\} = H_i x$, where $x' \in X_\sigma$. Note that by Lemma~\ref{invar} the orbit $ R_i x$ consists of idempotents. Since the set of idempotents is closed, $x'$ is an idempotent. Hence, the orbit $H_i x$ also consists of idempotents. The same holds for $H_\gamma$, since $H_\gamma$ is generated by $H_i$.
\end{proof}

\begin{lemma}\label{st:st}
Let $\beta$ be a face of the cone $\sigma$. The points in the orbit $O_\beta$ are fixed under $R_p$ if and only if $p \in \langle\beta \rangle$.
\end{lemma}

\begin{proof} Recall that $x \in O_\beta$ is equivalent to $\chi^u(x) = 0$ for all $u \notin \beta^\perp \cap \sigma^\vee$ and $\chi^u(x) \neq 0$ for all $u \in \beta^\perp \cap \sigma^\vee$. Note that a point~$x$ is fixed under $R_p$ if the equality $R_p \chi^u(x)  = \chi^u(x)$ holds for every $u \in \sigma^\vee$. Since $R_p \chi^u = t^{\langle p, u \rangle} \chi^u$, the equality is equivalent to $\langle p, u \rangle = 0$ for all $\chi^u(x) \neq 0$. In other words, $\langle p, u \rangle = 0$ for all $u \in \beta^\perp \cap \sigma^\vee$. Hence, $p \in \left(  \beta^\perp \cap \sigma^\vee \right)^\perp = \langle \beta \rangle$.
\end{proof}

\begin{lemma}\label{st:cl_gam}
The closure $\ol{E_\gamma}$ is contained in the $H_\gamma$-orbit of the point $x_{\cone(\tau, \gamma)}$.
\end{lemma}

\begin{proof} We use the induction on the number $s$ of primitive vectors $p_1, \ldots, p_s$ on the rays of the cone $\tau$ such that $p_1, \ldots, p_s \notin \gamma$, while $p_{s+1}, \ldots, p_k\in \gamma$. The base of induction $s = 0$ follows from item (1) of Theorem~\ref{idem_thm_1}. By the induction hypothesis $E_{\cone(\gamma, p_1)} \subseteq H_{\cone(\gamma, p_1)} x_{\cone(\tau, \gamma)}$. Take any $x \in E_\gamma$. Consider the $R_1$-orbit of the point $x$. Since $p_1 \notin \langle \gamma \rangle$, by Lemma~\ref{st:st} this orbit cannot consist of fixed points and therefore has dimension one. Recall that $\ol{R_1 x} = R_1 x \cup \{x'\} = H_1 x$, where $x' \in O_{\cone(\gamma, p_1)}$. By Lemma~\ref{st:h-idem}, $x'$ is an idempotent, and therefore lies in $H_{\cone(\gamma, p_1)} x_{\cone(\tau, \gamma)}$ by the induction hypothesis. Hence, we have $x = h_1 x' = h_1 h x_{\cone(\tau, \gamma)}$, where $h_1 \in H_1$, $h \in H_{\cone(\gamma, p_1)} = \langle H_2, \ldots, H_s \rangle$. Therefore, $E_\gamma \subseteq H_\gamma x_{\cone(\tau, \gamma)}$. Moreover, $H_\gamma x_{\cone(\tau, \gamma)}$ is an orbit of a unipotent group, hence closed, which implies $\ol{E_\gamma} \subseteq H_\gamma x_{\cone(\tau, \gamma)}$.
\end{proof}

\begin{lemma}\label{st:incl}
We have the inclusion $H_\gamma x_{\cone(\tau, \gamma)} \subseteq \bigcup_{ \{i_1, \ldots, i_r\} \subseteq  \{1, \ldots, s\}} E_{\cone(\gamma, p_{i_1}, \ldots, p_{i_r})}$. As a consequence, $\ol{E_\gamma} \subseteq \bigcup_{ \{i_1, \ldots, i_r\} \subseteq \{1, \ldots, s\}} E_{\cone(\gamma, p_{i_1}, \ldots, p_{i_r})}$.
\end{lemma}

\begin{proof} Recall that $H_\gamma$ is generated by the one-dimensional unipotent groups $H_1, \ldots, H_s$. Moreover, any orbit of $H_i$ passes through exactly two $\TT$-orbits, see~\eqref{idempHe_lemeq}. Since we are interested in the orbit of the point $x_{\cone(\tau, \gamma)}$, the set $H_\gamma x_{\cone(\tau, \gamma)}$ may contain only points from orbits of the form $ O_{\cone(\gamma, p_{i_1}, \ldots, p_{i_r})}$, where $ \{i_1, \ldots, i_r\} \subseteq \{1, \ldots, s\}$. 
It follows that we have the inclusion $H_\gamma x_{\cone(\tau, \gamma)} \subseteq \bigcup_{ \{i_1, \ldots, i_r\} \subseteq \{1, \ldots, s\}} O_{\cone(\gamma, p_{i_1}, \ldots, p_{i_r})}$. Moreover, by Lemma~\ref{st:h-idem}, $H_\gamma x_{\cone(\tau, \gamma)}$ consists of idempotents. Hence $H_\gamma x_{\cone(\tau, \gamma)} \subseteq \bigcup_{\{i_1, \ldots, i_r\} \subseteq \{1, \ldots, s\}} E_{\cone(\gamma, p_{i_1}, \ldots, p_{i_r})}$.
	
By Lemma~\ref{st:cl_gam} we have $\ol{E_\gamma} \subseteq H_\gamma x_{\cone(\tau, \gamma)}$, which proves the second statement.
\end{proof}

\begin{lemma}\label{one_orb}
The set $E_\gamma$ is a single $R_\tau$-orbit.
\end{lemma}

\begin{proof}
We prove this by induction on the number $s$ of primitive vectors $p_1, \ldots, p_s$ on the rays of the cone $\tau$, such that $p_1, \ldots, p_s \notin \gamma$, while $p_{s+1}, \ldots, p_k \in \gamma$. The base of induction $s = 0$ follows from item (1) of Theorem~\ref{idem_thm_1}. By the induction hypothesis, all idempotents in the orbit $O_{\cone(\gamma, p_1)}$ form a single $R_\tau$-orbit. Take arbitrary $x, y \in E_\gamma$. Consider the $R_1$-orbits of the points $x$ and $y$. Since $p_1 \notin \langle \gamma \rangle$, it follows from Lemma~\ref{st:st} that these orbits cannot consist of fixed points and therefore are one-dimensional. Recall that by the results of~\cite{AKZ2012} we have $\ol{R_1 x} = R_1 x \cup \{x'\} = H_1 x$, where $x' \in O_{\cone(\gamma, p_1)}$. Similarly, define $y' \in O_{\cone(\gamma, p_1)}$. By Lemma~\ref{st:h-idem}, $x', y'$ are idempotents, and hence lie in the same $R_\tau$-orbit by the induction hypothesis. Thus $y = h_1 t h_2 x$, where $h_1, h_2 \in H_1$, $t \in R_\tau$. Recall that $H_1$ is normalized by the torus $T \subseteq \TT$, i.e. $t^{-1} h_1 t = h_1' \in H_1$. Hence, $y = t t^{-1} h_1 t h_2 x \in tH_1x$,
so $t^{-1} y \in H_1 x$. On the other hand, $t^{-1} y \in O_\gamma$, since $t^{-1} \in R_\tau \subseteq \TT$, and $O_\gamma$ is a $\TT$-orbit. Therefore, $t^{-1} y \in H_1 x \cap O_\gamma = R_1x$. That is, $t^{-1} y = t_1 x$ for some $t_1 \in R_1 \subseteq R_\tau$. It follows that $y = t t_1 x$, where $t, t_1 \in R_\tau$.

$$
\begin{tikzpicture}
	\draw[thick, rounded corners=10pt] (0,0) rectangle (5, 4);
	\draw[thick, rounded corners=10pt] (5.2,0) rectangle (6.2, 4);
	\draw (2.5, -0.5) node{$O_\gamma$};
	\draw (6.3, -0.5) node{$O_{\cone(\gamma, \tau)}$};
	
	\draw[thick, gray!50!white] (0.5, 1) node[above]{$R_1x$} -- (4.5, 1);
	\draw[thick, gray!50!white] (0.5, 3) node[below]{$R_1y$} -- (4.5, 3);
	
	\draw [thick,->, orange] (3.7, 1) --node[above]{$h_2$} (5.5, 1);
	\draw [thick,->, orange] (5.7, 1.2) --node[right]{$t$} (5.7, 2.8);
	\draw [thick,->, orange] (5.5, 3) -- node[below]{$h_1$}(2.2, 3);
	\draw [thick,->, orange] (2, 2.8) --node[left]{$t^{-1}$} (2, 1.2);
	
	\filldraw (3.5, 1) node[below]{$x$} circle (2pt);
	\filldraw (5.7, 1) node[below]{$x'$} circle (2pt);
	\filldraw (2, 3) node[above]{$y$} circle (2pt);
	\filldraw (5.7, 3) node[above]{$y'$} circle (2pt);
	\filldraw (2, 1) node[below]{$t^{-1}y$} circle (2pt);
\end{tikzpicture}
$$
\end{proof}

\begin{lemma}\label{st:incl2}
We have the inclusion $E_{\cone(\gamma, p_{i_1}, \ldots, p_{i_r})} \subseteq \ol{E_\gamma}$.
\end{lemma}

\begin{proof}
First prove $E_{\cone(\gamma, p_{i_1})} \subseteq \ol{E_\gamma}$. Since by Lemma~\ref{invar} the set $E_\gamma$ is invariant under the action of $R_\tau$, its closure $\ol{E_\gamma}$ is also $R_\tau$-invariant. Moreover, the orbit $O_{\cone(\gamma, p_{i_1})}$ is also $R_\tau$-invariant, since $R_\tau$ is a subtorus of $\TT$. Therefore, the set $\ol{E_\gamma} \cap O_{\cone(\gamma, p_{i_1})}$ is $R_\tau$-invariant, i.e. it is a union of $R_\tau$-orbits. On the other hand, $\ol{E_\gamma} \cap O_{\cone(\gamma, p_{i_1})} \subseteq E_{\cone(\gamma, p_{i_1})}$, since $E_{\cone(\gamma, p_{i_1})}$ contains all idempotents lying in the orbit $O_{\cone(\gamma, p_{i_1})}$. Since $E_{\cone(\gamma, p_{i_1})}$ consists of exactly one $R_\tau$-orbit by Lemma~\ref{one_orb}, we obtain the equality $\ol{E_\gamma} \cap O_{\cone(\gamma, p_{i_1})} = E_{\cone(\gamma, p_{i_1})}$. Thus $E_{\cone(\gamma, p_{i_1})} \subseteq  \ol{E_\gamma}$. Therefore, $E_{\cone(\gamma, p_{i_1}, \ldots, p_{i_r})} \subseteq \ol{E_{\cone(\gamma, p_{i_1}, \ldots, p_{i_r})}} \subseteq \ldots  \subseteq \ol{E_{\cone(\gamma, p_{i_1}, p_{i_2})}} \subseteq \ol{E_{\cone(\gamma, p_{i_1})}} \subseteq  \ol{E_\gamma}$. 
\end{proof}

\begin{proof}[Proof of Theorem~\ref{idem_thm_2}]
Recall that the condition $E_\gamma \neq \varnothing$ holds for cases (1) and (3) of Theorem~\ref{idem_thm_1}, with case (1) being a special case of (3).  
Item (1) follows directly from Lemmas~\ref{st:incl} and~\ref{st:incl2}. Next, note that $x_\gamma \in E_\gamma$ by item (3) of Theorem~\ref{idem_thm_1}, and item (2) of the theorem follows from Lemma~\ref{one_orb}. Finally, item (3) of the theorem follows from Lemmas~\ref{st:cl_gam}, \ref{st:incl}, and item (1) of the theorem.   
\end{proof}

\section{The center of an active monoid}\label{sec:center}

Denote the center of the monoid by $Z(X) = \{x \in X \mid \forall y \in X\;\; x*y = y*x\}$. The group $G$ of invertible elements is open and dense in $X$, hence the center can be described as follows: \[Z(X) = \{x \in X \mid \forall y \in G \;\;x = y*x *y^{-1}\}.\] 

Let $X$ be a root monoid with the group of invertible elements $\GG$. Recall the notation introduced above: $\ol{p} = (p_1, \ldots, p_k)$ is the set of primitive vectors on the rays of the cone $\tau$, $\langle \ol{p}, u \rangle = (\langle p_1, u \rangle, \ldots, \langle p_k, u \rangle)$, $\ol{i} = (i_1, \ldots, i_k) \in \ZZ^k$, and ${ \langle \ol{p}, u \rangle \choose \ol{i} }  =  { \langle p_1, u \rangle \choose i_1} \cdot \ldots \cdot  { \langle p_k, u \rangle \choose i_k}$. Additionally, denote $(-1)^{\ol{i}} := (-1)^{i_1 + \ldots + i_k}$.

\begin{lemma}
Let $y \in \GG$. Then the inverse element $y^{-1}$ is determined by the conditions
\begin{equation}\label{eq_inv}
\chi^u(y^{-1}) = (-1)^{\langle \ol{p},u\rangle}
\chi^{-u - \langle \ol{p}, u \rangle (\ol{e_1} + \ol{e_2})}(y) \text{ for all } u\in S_\tau,
\end{equation}
where $\langle \ol{p}, u \rangle (\ol{e_1} + \ol{e_2}) = \sum \limits_{r = 1}^{k} \langle p_r, u \rangle (e_1^{(r)} + e_2^{(r)})$.
\end{lemma}

\begin{proof}
Note that formula~\eqref{eq_inv} defines a certain semigroup homomorphism $S_\tau \to \KK$, i.e. a point $z \in \GG$. We need to check that $z = y^{-1}$, that is, $y*z = z*y = x_\tau$. Recall that
\begin{equation}\label{mm}
	\chi^u (y * z) = \sum \limits_{\ol{i} + \ol{j} = \langle \ol{p}, u \rangle } { \langle \ol{p}, u \rangle \choose \ol{i} } \chi^{u + \ol{i}\cdot \ol{e_2}}(y) \cdot  \chi^{u + \ol{j}\cdot \ol{e_1}} (z).
\end{equation}
Denote $u' = u + \ol{j}\cdot \ol{e_1} = u + j_1 e_1^{(1)} + j_2 e_1^{(2)} + \cdots + j_k e_1^{(k)}$. Then $\langle p_r, u' \rangle = \langle p_r, u \rangle - j_r = i_r$ since $\ol{i} + \ol{j} = \langle \ol{p}, u \rangle$. It follows that $-u' - \langle \ol{p}, u' \rangle (\ol{e_1} + \ol{e_2}) = -u -\ol{i}\ol{e_2} - \langle \ol{p}, u \rangle\ol{e_1}$. Using these equalities, compute $\chi^u(y * z)$ by formula~\eqref{mm}, substituting $u'$ instead of $u$ in~\eqref{eq_inv} which defines $z$:
\begin{align*}
 	&\sum \limits_{\ol{i} + \ol{j} = \langle \ol{p}, u \rangle } { \langle \ol{p}, u \rangle \choose \ol{i} } \chi^{u + \ol{i}\cdot \ol{e_2}}(y) \cdot  \chi^{u + \ol{j}\cdot \ol{e_1}} (z) 
 	= \sum \limits_{\ol{i} + \ol{j} = \langle \ol{p}, u \rangle } { \langle \ol{p}, u \rangle \choose \ol{i} } \chi^{u + \ol{i}\cdot \ol{e_2}}(y) \cdot  (-1)^{\ol{i}} \chi^{-u -\ol{i}\ol{e_2} - \langle \ol{p}, u \rangle\ol{e_1}}(y)  = \\
 	&=\sum \limits_{\ol{i} + \ol{j} = \langle \ol{p}, u \rangle } { \langle \ol{p}, u \rangle \choose \ol{i} } \cdot  (-1)^{\ol{i}} \chi^{-u -\ol{i}\ol{e_2} - \langle \ol{p}, u \rangle\ol{e_1} + u + \ol{i}\cdot \ol{e_2}}(y) = 
 	\chi^{-\langle \ol{p}, u \rangle\ol{e_1} }(y) \sum \limits_{\ol{i} + \ol{j} = \langle \ol{p}, u \rangle } { \langle \ol{p}, u \rangle \choose \ol{i} } \cdot  (-1)^{\ol{i}}.
 \end{align*}

We obtain
 \begin{equation}
 	\chi^u (y * z)  = \chi^{-\langle \ol{p}, u \rangle\ol{e_1} }(y) \sum \limits_{\ol{i} + \ol{j} = \langle \ol{p}, u \rangle} {\langle \ol{p}, u \rangle \choose \ol{i} } \cdot (-1)^{\ol{i}}.
 \end{equation} 

Recall that the neutral element $x_\tau$ is the distinguished point for the cone $\tau$, namely
\begin{equation}
	\begin{cases} 
			\chi^u(x_\tau) = 1, \text{ if } u \in \tau^\perp\\
		\chi^u(x_\tau) = 0 \text{ otherwise} .
	\end{cases}
\end{equation} 

First, check that $\chi^u (y * z)  = 1$ if $u \in \tau^\perp$. In this case $\langle \ol{p}, u \rangle = (0, \ldots, 0)$, which follows that $\chi^{-\langle \ol{p}, u \rangle\ol{e_1} }(y) = \chi^0(y) = 1$, and the sum reduces to a single nonzero term corresponding to $\ol{i} = \ol{0}$, equal to $1$.  

Now check that $\chi^u (y * z)  = 0$ if $u \notin \tau^\perp$. In this case there exists $r$ with $\langle p_r, u \rangle \neq 0$. Let us prove that
\begin{equation}
	\sum \limits_{\ol{i} + \ol{j} = \langle \ol{p}, u \rangle } { \langle \ol{p}, u \rangle \choose \ol{i} } \cdot  (-1)^{\ol{i}} = 0.
\end{equation}
Indeed,
\begin{gather*}
\sum \limits_{\ol{i} + \ol{j} = \langle \ol{p}, u \rangle } { \langle \ol{p}, u \rangle \choose \ol{i} } \cdot  (-1)^{\ol{i}} = \sum \limits_{\ol{i} + \ol{j} = \langle \ol{p}, u \rangle } { \langle p_1, u \rangle \choose i_1 } \cdot  (-1)^{i_1} \ldots { \langle p_k, u \rangle \choose i_k } \cdot  (-1)^{i_k} =\\= \prod \limits_{s = 1}^{k} \biggl(\sum \limits_{i_s + j_s = \langle p_s, u \rangle } { \langle p_s, u \rangle \choose i_s } \cdot  (-1)^{i_s}\biggr).
\end{gather*}
Consider the factor $\sum \limits_{i_r + j_r = \langle p_r, u \rangle} {\langle p_r, u \rangle \choose i_r} \cdot (-1)^{i_r}$, for which $\langle p_r, u \rangle \neq 0$. Note that this is an alternating sum of binomial coefficients, which equals $0$. Hence the entire product vanishes.
\end{proof}

The following theorem provides explicit equations for the center of an active monoid.

\begin{theorem} \label{thm:z}
In the notation of Theorem~\ref{thm:act}, the center of the active monoid $X$ is described as follows:
$$Z(X) = \ol{O_\tau} \cap \{\chi^{u + e_1^{(r)}} =  \chi^{u + e_2^{(r)}} \mid (r,u) \in I\},$$ 
where $I = \{(r,u) \mid r = 1, \ldots, k, \, u \in S_\sigma: \; \langle p_j, u \rangle = \delta_{jr} \text{ for all } j = 1, \ldots, k\}$.
\end{theorem}

\begin{proof}
Let us compute $\chi^u(y*x*y^{-1})$, where $x\in X$, $y \in \GG$. Take $u \in S_\sigma$, then
\begin{align*}
	&\chi^u(y*x*y^{-1}) = \sum \limits_{\ol{i} + \ol{j} = \langle \ol{p}, u \rangle } { \langle \ol{p}, u \rangle \choose \ol{i} } \chi^{u + \ol{i}\cdot \ol{e_2}}(y*x) \cdot  \chi^{u + \ol{j}\cdot \ol{e_1}} (y^{-1}) = \\
	&=\sum \limits_{\ol{i} + \ol{j} = \langle \ol{p}, u \rangle } { \langle \ol{p}, u \rangle \choose \ol{i} } (-1)^{\ol{i}} \chi^{u + \ol{i}\cdot \ol{e_2}}(y*x) \cdot  \chi^{-u - \ol{i}\cdot \ol{e_2} - \langle \ol{p}, u \rangle \ol{e_1}} (y) =\\
	&= \sum \limits_{\ol{i} + \ol{j} = \langle \ol{p}, u \rangle } { \langle \ol{p}, u \rangle \choose \ol{i} } (-1)^{\ol{i}} \left(
	\sum \limits_{\ol{t} + \ol{l} = \ol{j}} { \ol{j} \choose \ol{t} } \chi^{u + \ol{i}\cdot \ol{e_2} +  \ol{t}\cdot \ol{e_2}}(y) \cdot  \chi^{u + \ol{i}\cdot \ol{e_2} + \ol{l}\cdot \ol{e_1}} (x)
	 \right)  
	\cdot \chi^{-u - \ol{i}\cdot \ol{e_2} - \langle \ol{p}, u \rangle \ol{e_1}} (y)=\\
	&= \sum \limits_{\ol{i} + \ol{t} + \ol{l} = \langle \ol{p}, u \rangle } { \langle \ol{p}, u \rangle \choose \ol{i} } (-1)^{\ol{i}}  {\langle \ol{p}, u \rangle - \ol{i}  \choose \ol{t} } \chi^{u + \ol{i}\cdot \ol{e_2} +  \ol{l}\cdot \ol{e_1}}(x) \cdot  \chi^{\ol{t}\cdot \ol{e_2} - \langle \ol{p}, u \rangle \ol{e_1} } (y)=\\
	&= \sum \limits_{\ol{t} = \ol{0}}^{\langle \ol{p}, u \rangle}  { \langle \ol{p}, u \rangle \choose \ol{t} }  \left( \sum \limits_{\ol{i} + \ol{l} = \langle \ol{p}, u \rangle - \ol{t}} (-1)^{\ol{i}} { \langle \ol{p}, u \rangle - \ol{t} \choose \ol{i} }  \cdot  \chi^{u + \ol{i}\cdot \ol{e_2} + \ol{l}\cdot \ol{e_1}} (x) \right) 
	\cdot \chi^{\ol{t}\cdot \ol{e_2} - \langle \ol{p}, u \rangle \ol{e_1}} (y)=\\
	&= \sum \limits_{\ol{t} = \ol{0}}^{\langle \ol{p}, u \rangle}  { \langle \ol{p}, u \rangle \choose \ol{t} }  
	\chi^u \cdot \prod\limits_{r = 1}^{k}\left( \chi^{e_1^{(r)}} - \chi^{e_2^{(r)}}\right)^{\langle p_r , u\rangle - t_r} (x)
	\cdot  \chi^{\ol{t}\cdot \ol{e_2} - \langle \ol{p}, u \rangle \ol{e_1}} (y).
\end{align*}

Consider the term of the above sum corresponding to $\ol{t} = \langle \ol{p}, u \rangle$. In this case ${ \langle \ol{p}, u \rangle \choose \ol{t} } = 1$ and $\langle p_r , u\rangle - t_r = 0$. Therefore, this term equals $\chi^u(x) \cdot \chi^{\langle \ol{p}, u \rangle (\ol{e_2} - \ol{e_1})}(y)$. Recall that in order to check whether ${x \in Z(X)}$, it is necessary to verify $\chi^u(yxy^{-1}) - \chi^u(x) = 0$ for all $u \in \sigma^\vee$, $y \in \GG = G(X)$. Rewriting the equality $\chi^u(yxy^{-1}) - \chi^u(x) = 0$ using formula~\eqref{mm}, we obtain:
\begin{equation}
\label{form:z}
\begin{aligned}
\chi^u(x) \Bigl(&\chi^{\langle \ol{p}, u \rangle (\ol{e_2} - \ol{e_1})}(y) - 1 \Bigr) + \\ 
&+ \sum \limits_{\ol{0} \le \ol{t} < \langle \ol{p}, u \rangle} { \langle \ol{p}, u \rangle \choose \ol{t} }  
\chi^u \cdot \prod\limits_{r = 1}^{k}\left(\chi^{e_1^{(r)}} - \chi^{e_2^{(r)}}\right)^{\langle p_r , u\rangle - t_r} (x)
\cdot  \chi^{\ol{t}\cdot \ol{e_2} - \langle \ol{p}, u \rangle \ol{e_1}} (y) = 0.
\end{aligned}
\end{equation}

Fix $x$ and check whether expression~\eqref{form:z} vanishes for arbitrary $y \in \GG$. Recall that $\GG = X_\tau$ and $\KK[X_\tau] = \bigoplus_{u \in S_\tau} \KK \chi^u$. Let $\langle \ol{p}, u \rangle \neq \ol{0}$. Then, since $\GG$ is active, the coordinates ${\ol{e_2} - \ol{e_1}}$ are linearly independent vectors, i.e., their nontrivial linear combination $\langle \ol{p}, u \rangle (\ol{e_2} - \ol{e_1})$ is nonzero. Then the functions $\chi^{\langle \ol{p}, u \rangle (\ol{e_2} - \ol{e_1})}(y) - \chi^0(y)$ and $\chi^{\ol{t}\cdot \ol{e_2} - \langle \ol{p}, u \rangle \ol{e_1}} (y)$ for $\ol{t} \neq \langle \ol{p}, u \rangle$ are linearly independent in $\KK[X_\tau]$, since they lie in different homogeneous components. Now suppose $\langle \ol{p}, u \rangle = \ol{0}$. Then $\chi^{\langle \ol{p}, u \rangle (\ol{e_2} - \ol{e_1})}(y) - \chi^0(y) = 0$, while the remaining $\chi^{\ol{t}\cdot \ol{e_2} - \langle \ol{p}, u \rangle \ol{e_1}} (y)$ for $\ol{t} \neq \langle \ol{p}, u \rangle$ are linearly independent in $\KK[X_\tau]$. Thus, in the linear combination~\eqref{form:z}, all coefficients depending on $x$ have to be equal~$0$. Hence, $x \in Z(X)$ if and only if
\begin{equation}
   	\begin{cases}
   		\chi^u(x) = 0, \text{ если } \langle \ol{p}, u \rangle \neq \ol{0}\\
   		 \chi^u \cdot \prod\limits_{r = 1}^{k}\left( \chi^{e_1^{(r)}} - \chi^{e_2^{(r)}}\right)^{\langle p_r , u\rangle - t_r} (x) = 0, \text{ если } \ol{0} \le \ol{t} < \langle \ol{p}, u \rangle
   	\end{cases}
\end{equation}

The condition $\chi^u(x) = 0$ for $\langle \ol{p}, u \rangle \neq \ol{0}$ is equivalent to $x \in \overline{O_\tau}$. Let us look at the second condition. Note that if $\langle p_r , u\rangle = 0$ for some $r = 1, \ldots, k$, then the corresponding factor $\left( \chi^{e_1^{(r)}} - \chi^{e_2^{(r)}}\right)^{\langle p_r, u\rangle - t_r}$ can be ignored, since $\left( \chi^{e_1^{(r)}} - \chi^{e_2^{(r)}}\right)^{\langle p_r, u\rangle - t_r} = \left( \chi^{e_1^{(r)}} - \chi^{e_2^{(r)}}\right)^0 = 1$. Further we consider only those $r$ for which $\langle p_r , u\rangle > 0$. Observe that the second condition must hold, in particular, for $\ol{t} = \langle \ol{p} , u\rangle - \ol{s}$, where $\ol{s} = (0, \ldots, 0, \underset{r}{1}, 0, \ldots, 0)$. In other words, $\chi^{u + e_1^{(r)}}(x) =  \chi^{u + e_2^{(r)}}(x)$.  To complete the proof of Theorem~\ref{thm:z}, it remains to prove the following lemma.

\begin{lemma}
Suppose $\chi^{u + e_1^{(r)}}(x) =  \chi^{u + e_2^{(r)}}(x)$ for every $r=1, \ldots, k$ and every $u \in S_\sigma$ such that $\langle p_j, u \rangle = \delta_{jr}$ for $j = 1, \ldots, k$. Then \[\chi^u \cdot \prod\limits_{r = 1}^{k}\left( \chi^{e_1^{(r)}} - \chi^{e_2^{(r)}}\right)^{\langle p_r , u\rangle - t_r} (x) = 0\] for all $u \in S_\sigma$ such that $\ol{0} \le \ol{t} < \langle \ol{p}, u \rangle$.
\end{lemma}

\begin{proof}
We argue by induction on the sum of the coordinates of the vector $\langle \ol{p}, u \rangle$.  

If the sum of the coordinates of $\langle \ol{p}, u \rangle$ equals $1$, then $\ol{t}=\ol{0}$ and the claim follows from the assumption of the lemma.  

Set $\ol{s} = \langle \ol{p} , u\rangle - \ol{t}$ and $P:= \prod\limits_{r = 2}^{k}\left( \chi^{e_1^{(r)}} - \chi^{e_2^{(r)}}\right)^{s_r} $. 

Suppose $s_1 \neq 0$. Denote $u' = u + e_1^{(1)}$ and $u'' = u + e_2^{(1)}$. Then
\begin{align*} 
	&\chi^u \cdot \prod\limits_{r = 1}^{k}\left( \chi^{e_1^{(r)}} - \chi^{e_2^{(r)}}\right)^{\langle p_r , u\rangle - t_r}  \!\!= 
	\chi^u \cdot \left( \chi^{e_1^{(1)}} \!\!- \chi^{e_2^{(1)}}\right)^{s_1} \!\!\cdot P = \\
	&= P \cdot \sum\limits_{l = 0}^{s_1} {s_1 \choose l} (-1)^l \chi^{u+ le_2^{(1)} + (s_1 - l)e_1^{(1)}}=\\
	&= P \cdot \sum\limits_{l = 0}^{s_1-1} {s_1-1 \choose l} (-1)^l \chi^{u+ le_2^{(1)} + (s_1 - l)e_1^{(1)}} + {s_1-1 \choose l-1} (-1)^l \chi^{u+ le_2^{(1)} + (s_1 - l)e_1^{(1)}}=\\
	&= P \cdot \sum\limits_{l = 0}^{s_1-1} {s_1-1 \choose l} (-1)^l \chi^{u' + le_2^{(1)} + (s_1 -1 - l)e_1^{(1)}} - \\
	&\hspace{4cm}- P \cdot \sum\limits_{l = 0}^{s_1-1} {s_1-1 \choose l-1} (-1)^{l-1} \chi^{u''+  (l - 1)e_2^{(1)} + (s_1 - l)e_1^{(1)}}=\\
	&= \chi^{u'} \cdot \prod\limits_{r = 1}^{k}\left(\chi^{e_1^{(r)}} - \chi^{e_2^{(r)}}\right)^{\langle p_r , u'\rangle - t_r}
	- \chi^{u''} \cdot \prod\limits_{r = 1}^{k}\left( \chi^{e_1^{(r)}} - \chi^{e_2^{(r)}}\right)^{\langle p_r , u''\rangle - t_r}.
\end{align*}
In intermediate computations we assume that the binomial coefficient ${i \choose j}$ equals $0$ if $j < 0$. Further note that the first term vanishes when evaluated at $x$ according to the induction hypothesis for $u' = u + e_1^{(1)}$, and the second term vanishes by the induction hypothesis for $u'' = u + e_2^{(1)}$. Since we assumed $s_1 > 0$, we have $\langle p_1, u \rangle > 0$, hence $u', u'' \in S_\sigma$.
\end{proof}
Thus, Theorem~\ref{thm:z} is proved.
\end{proof}

\begin{remark}
If the group of invertible elements $G_\chi$ is not an active semidirect product, then the set
\[\ol{O_\tau} = \{x \in X \mid \chi^u(x) = 0 \text{ for those } u \in S_\sigma, \text{ such that } \langle \ol{p}, u \rangle \ne 0\}\] 
in the formulation of Theorem~\ref{thm:z} must be enlarged and replaced by
\[\{x \in X \mid \chi^u(x) = 0 \text{ for those } u \in S_\sigma, \text{ such that } \langle \ol{p}, u \rangle (\ol{e_2} - \ol{e_1}) \ne 0\}.\]
In particular, if $G$ is commutative, then $\ol{e_1} = \ol{e_2}$, whence for any $u$ we have $\langle\ol{p}, u \rangle (\ol{e_2} - \ol{e_1}) = 0$; so the set of equations on $x \in X$ is empty and $Z(X) = X$.
\end{remark}

\begin{example}
Let us compute the center of the monoid on the cylinder ${X = \{x_1x_2=x_4x_5\}\subseteq \AA^5}$ in Example~\ref{example2}. According to Theorem~\ref{thm:z}, the center $Z(X)$ is given by equations of the form $\chi^{u + e_1^{(r)}} =  \chi^{u + e_2^{(r)}}$ in the closure of the orbit $\overline{O_\tau}$, where $r = 1, 2$, $u \in S_\sigma$ and $\langle p_j, u \rangle = \delta_{jr}$ for $j = 1, 2$.  

The closure of the orbit $\overline{O_\tau}$ consists of points $x \in X_\sigma$ satisfying equations $\chi^u = 0$ for all $u \in S_\sigma \setminus \tau^\perp$. Note that \[S_\sigma \setminus \tau^\perp = (q_1 + S_\sigma) \cup (q_2 + S_\sigma) \cup (q_5 + S_\sigma).\]
Hence $\overline{O_\tau}$ is defined by equations $\chi^{q_1} = \chi^{q_2} = \chi^{q_5} = 0$, i.e., $x_1=x_2=x_5=0$.  

For $r=1$, we need to consider $u \in S_\sigma$ such that $\langle p_1, u\rangle = 1$ and $\langle p_2, u\rangle = 0$, that is, $u \in S_1 = \{(1, 0, a, b) \mid a,b \ge 0\}$. Note that $S_1 \subseteq q_1 + S_\sigma$, so for $r=1$ it suffices to write the equation $\chi^{q_1 + e_1^{(1)}} =  \chi^{q_1 + e_2^{(1)}}$. We obtain $\chi^{(0,0,a_1,b_1)} = \chi^{(0,0,a_2,b_2)}$, i.e., $x_3^{a_1}x_4^{b_1} = x_3^{a_2}x_4^{b_2}$. Similarly, for $r=2$ we obtain the equation $x_3^{c_1}x_4^{d_1} = x_3^{c_2}x_4^{d_2}$. 

Thus, $Z(X) \subseteq X$ is given by the system
\[x_1=x_2=x_5=0, \quad x_3^{a_1}x_4^{b_1} = x_3^{a_2}x_4^{b_2}, \quad x_3^{c_1}x_4^{d_1} = x_3^{c_2}x_4^{d_2}.\]

Note that this system implies $x_3^{a_1+c_1}x_4^{b_1+d_1-1} = x_3^{a_2+c_2}x_4^{b_2+d_2-1}$, since ${b_1+d_1-1\ge 1}$ and $b_2+d_2-1 \ge 1$. Therefore, by direct substitution of $x \in Z(X)$ that we found into $x*y$ and $y*x$ according to formula~\eqref{eq_mult_example2}, we obtain the same coordinates.
\end{example}


\begin{thebibliography}{99}
		\bibitem[AAZ]{AAZ} {\sc Ivan Arzhantsev, Roman Avdeev, and Yulia Zaitseva}, {\em Algebraic monoid structures on the affine 3-space.} https://arxiv.org/abs/2412.09559, 14 pages (2024)
		%
		\bibitem[ABZ]{ABZ} {\sc Ivan Arzhantsev, Sergey Bragin, Yulia Zaitseva}, {\em Commutative algebraic monoid structures on affine spaces.} Commun. Contemp. Math.~22, no.~8, article~1950064 (2020)
		%
		\bibitem[AKZ]{AKZ2012} {\sc Ivan Arzhantsev, Karine Kuyumzhiyan, and Mikhail Zaidenberg.} {\em Infinite transitivity, finite generation, and Demazure roots.} Adv. Math.~351, 1-32 (2019)
		%
		\bibitem[Bil]{Bil} {\sc Boris Bilich}, {\em Classification of noncommutative monoid structures on normal affine surfaces.}  Proc. Amer. Math. Soc.~150, no.~10, 4129-4144 (2022)
		%
		\bibitem[Br-1]{Br-1}
		{\sc Michel Brion}, {\em On Algebraic Semigroups and Monoids.} In: Algebraic Monoids, Group Embeddings, and Algebraic Combinatorics. Fields Institute Communications, vol.~71, Springer, New York, 1-54 (2014)
		%
		\bibitem[Br-2]{Br-2}
		{\sc Michel Brion}, {\em On algebraic semigroups and monoids, II.} Semigroup Forum 88, no.~1, 250-272 (2014)
		%
		\bibitem[CLS]{F} {\sc David Cox, John Little, Henry Schenck}, {\em Toric Varieties.} Grad. Stud. Math. 124, Amer. Math. Soc., Providence, RI (2011)
		%
		\bibitem[DZ]{DZ} {\sc Sergey Dzhunusov, Yulia Zaitseva}, {\em Commutative algebraic monoid structures on affine surfaces.} Forum Math.~33, no.~1, 177-191 (2021)
		%
		\bibitem[N]{Neeb} {\sc Karl-Hermann Neeb}, {\em Toric varieties and algebraic monoids.} Seminar Sophus Lie 2, 159-187  (1992)
		%
		\bibitem[Re]{Re} {\sc Lex Renner }, \emph{Linear Algebraic Monoids}. Encyclopaedia Math. Sci.~134, Springer, Berlin (2005)
		%
		\bibitem[Pu]{Pu} {\sc Mohan Putcha} \emph{Linear Algebraic Monoids.} London Math. Soc. Lecture Note Ser., vol.~133, Cambridge Univ. Press, Cambridge (1988)
		%
		\bibitem[Ri1]{Ri1} {\sc Alvaro Rittatore}, {\em Algebraic monoids and group embeddings.} Transform. Groups 3, no.~4, 375-396 (1998)
		%
		\bibitem[Ri2]{Ri2} {\sc Alvaro Rittatore}, {\em  Algebraic monoids with affine unit group are affine.} Transform. Groups 12, no.~3, 601-605 (2007)
		%
		\bibitem[Vin]{Vin} {\sc Ernest Vinberg}, {\em On reductive algebraic semigroups.} In: Lie Groups and Lie Algebras, E.B. Dynkin
		Seminar, S. Gindikin and E. Vinberg, Editors, Amer. Math. Soc. Transl. 169, 145-182, Amer. Math. Soc. (1995)
		%
		\bibitem[YZ]{YZ} {\sc Yulia Zaitseva}, {\em Affine monoids of corank one}, Results Math. 79, no. 7, article 249 (2024)
	\end{thebibliography}
\end{document}